\documentclass[a4paper,12pt,oneside]{article}
\usepackage[utf8]{inputenc}
\usepackage[english]{babel}
\usepackage{amsmath}
\usepackage{amsthm}
\usepackage{amssymb}
\usepackage{mathtools}
\usepackage{tikz}
\usepackage{verbatim}
\usepackage{fancyvrb}
\usepackage{array}
\usepackage{adjustbox}
\usepackage{changepage}
\usepackage{subcaption}
\usepackage{breqn}
\usepackage{enumerate}
\usepackage{wrapfig}
\usepackage{bbm}
\usepackage{bm}
\usepackage{enumitem}
\usepackage{csquotes}
\usepackage{hyperref}
\usepackage{framed}
\theoremstyle{definition}
\newtheorem{defi}{Definition}[section]
\newtheorem{thm}[defi]{Theorem}
\newtheorem{lem}[defi]{Lemma}

\newtheorem{prop}[defi]{Proposition}

\newtheorem{cor}[defi]{Corollary}

\newcommand{\N}{\mathbb{N}}
\newcommand{\Z}{\mathbb{Z}}
\newcommand{\Q}{\mathbb{Q}}
\newcommand{\C}{\mathbb{C}}
\newcommand{\R}{\mathbb{R}}

\newcommand{\p}{\varphi}
\newcommand{\BO}{\mathcal{O}}
\newcommand{\sz}[1]{\left| \vec{#1} \right|}

\title{Applications of multiple orthogonal polynomials with hypergeometric moment generating functions}
\author{Thomas Wolfs\footnote{Department of Mathematics, KU Leuven, Belgium. E-mail adress: \texttt{thomas.wolfs[at]kuleuven.be}.}} 
\date{}

\begin{document}
\maketitle

\begin{abstract}
We investigate several families of multiple orthogonal polynomials associated with weights for which the moment generating functions are hypergeometric series with slightly varying parameters. The weights are supported on the unit interval, the positive real line, or the unit circle and the multiple orthogonal polynomials are generalizations of the Jacobi, Laguerre or Bessel orthogonal polynomials. We give explicit formulas for the type~I and type~II multiple orthogonal polynomials and study some of their properties. In particular, we describe the asymptotic distribution of the (scaled) zeros of the type~II multiple orthogonal polynomials via the free convolution. Essential to our overall approach is the use of the Mellin transform. Finally, we discuss two applications. First, we show that the multiple orthogonal polynomials appear naturally in the study of the squared singular values of (mixed) products of truncated unitary random matrices and Ginibre matrices. Secondly, we use the multiple orthogonal polynomials to simultaneously approximate certain hypergeometric series and to provide an explicit proof of their $\Q$-linear independence.
\medbreak


\end{abstract}

\section{Introduction}

Multiple orthogonal polynomials, compared to (regular) orthogonal polynomials, satisfy orthogonality conditions with respect to several weights $(w_1,\dots,w_r)$, instead of just one weight (see \cite[Chapter 23]{Ismail} and \cite[Chapter 4]{NikiSor} for an introduction). They depend on a multi-index $\vec{n}\in\Z_{\geq 0}^r$, of size $\sz{n}=n_1+\dots+n_r$, that determines the way in which the orthogonality conditions are distributed. Suppose that the weights are supported on $\Lambda\subset\C$ and that their moments are all finite. The type I multiple orthogonal polynomials are given by vectors of polynomials $(A_{\vec{n},1},\dots,A_{\vec{n},r})$, with $\deg A_{\vec{n},j}\leq n_j-1$, for which the type I function $F_{\vec{n}} = \sum_{j=1}^r A_{\vec{n},j} w_j $ satisfies the orthogonality conditions
    $$\int_{\Lambda} F_{\vec{n}}(x) x^k dx = 0,\quad k=0,\dots,\sz{n}-2.$$
The type II multiple orthogonal polynomials are polynomials $P_{\vec{n}}$ of degree at most $\sz{n}$ that satisfy the orthogonality conditions
    $$\int_{\Lambda} P_{\vec{n}}(x) x^k w_j(x) dx = 0,\quad k=0,\dots,n_j-1,\quad j=1,\dots,r.$$
The existence of non-zero type I functions and type II polynomials is always guaranteed. The associated multi-index is called normal if the type I and type II polynomials attain their maximal degree. In that case, the type I functions and type II polynomials are determined uniquely up to a scalar multiplication. Systems of weights in which all multi-indices are normal are called perfect \cite{Mahler}. Many examples of perfect systems (see, e.g., \cite{CoussVA}) arise as so-called Angelesco systems \cite{Ang} or AT-systems \cite[Section 5.7]{NikiSor}, or more specifically as Nikishin systems \cite{FidLop1,FidLop2}. 
\medbreak

We will investigate families of multiple orthogonal polynomials for which the moment generating functions of the weights are (scaled) hypergeometric series of the form
\begin{equation} \label{Intro_HS}
        f_0(z;\vec{a},\vec{b}) = \sum_{k\geq 0} \frac{\Gamma(k+\vec{a}+1)}{\Gamma(k+\vec{b}+1)} z^{k},
\end{equation}
in which we use parameters $\vec{a}\in(-1,\infty)^p$ and $\vec{b}\in(-1,\infty)^q$. To simplify notation, we used the convention that applying the gamma function $\Gamma(s)=\int_0^\infty e^{-x} x^{s-1} dx $ to a vector $\vec{s}\in(0,\infty)^r$ yields the product $\prod_{j=1}^r \Gamma(s_j)$ of the values of the gamma function applied to its components. This convention extends to Pochhammer symbols $(s)_k = \prod_{j=0}^{k-1} (s+j)$ in the natural way: $(\vec{s})_k = \prod_{j=1}^r (s_j)_k $. The weight $w_0(x;\vec{a},\vec{b})$ associated with the moment generating function $f_0(z;\vec{a},\vec{b})$ in \eqref{Intro_HS} has moments 
\begin{equation} \label{Intro_mom}
    m_k(\vec{a},\vec{b})=\int_{\Lambda} x^k w_0(x;\vec{a},\vec{b}) dx=\frac{\Gamma(k+\vec{a}+1)}{\Gamma(k+\vec{b}+1)}
\end{equation}
and its support $\Lambda$ will be the unit interval $[0,1]$, if $p=q$, the positive real line $[0,\infty)$, if $p>q$, or the unit circle $\{z\in\C:|z|\,=1\}$, if $p<q$. It will be natural to consider $r=\max\{p,q\}$ weights $w_j(x;\vec{a},\vec{b})$, $1\leq j\leq r$, that arise from $w_0(x;\vec{a},\vec{b})$ by slightly varying its parameters.
\medbreak

In the most simple cases, where $r=1$, the weight $w_0(x;\vec{a},\vec{b})$ will reduce to one of the three classical orthogonality weights. For an introduction to orthogonal polynomials, we refer to~\cite{Szego}. Taking $\vec{a}=(a_1)$ and $\vec{b}=(b_1)$, so $(p,q)=(1,1)$, leads to the beta density $\mathcal{B}^{a_1,b_1}(x) = x^{a_1}(1-x)^{b_1-a_1-1}/\Gamma(b_1-a_1)$ on $[0,1]$ (assuming $a_1<b_1$). The associated orthogonal polynomials are the Jacobi polynomials $P_{n}^{(a_1,b_1-a_1-1)}(x) = {}_2F_1(-n,n+b_1;a_1+1;x)$. The gamma density $\mathcal{G}^{a_1}(x)=x^{a_1}e^{-x}$ on $[0,\infty)$ corresponds to the case $(p,q)=(1,0)$ with $\vec{a}=(a_1)$ and $\vec{b}=()$. In this case, the orthogonal polynomials are the Laguerre polynomials $L_n^{(a_1)}(x)={}_1F_1(-n;a_1+1;x)$. Finally, if $(p,q)=(0,1)$, we need to move to the complex plane and have to consider the weight $f_0(1/z;-,b_1)/z$ on $\{z\in\C:|z|\,=1\}$. The associated orthogonal polynomials are the Bessel polynomials $B_n^{(b_1-1)}(z)={}_2F_0(-n,n+b_1;-;z)$. The systems of weights that we will consider can therefore be seen as generalizations of the classical settings of Jacobi (Section 2), Laguerre (Section 3) and Bessel (Section 4).

Some special cases of the multiple orthogonal polynomials have been studied before; the Jacobi-Piñeiro polynomials, orthogonal with respect to $x^{\alpha_j}(1-x)^{\beta}$ for $1\leq j\leq r$, with $\beta=0$ in \cite{SmetVA} correspond to a subcase of $(p,q)=(r,r)$, the multiple orthogonal polynomials associated with Meijer-$G$ functions in \cite{KuijlaarsZhang} correspond to $(p,q)=(r,0)$, with $K$-Bessel functions in \cite{VAYakubovich} to $(p,q)=(2,0)$, with confluent hypergeometric functions in \cite{LimaLoureiro1} to $(p,q)=(2,1)$, with the exponential integral in \cite{VAWolfs} to a subcase of $(p,q)=(2,1)$ and with Gauss' hypergeometric functions in \cite{LimaLoureiro2} to $(p,q)=(2,2)$. 
The general $(p,q)$-case was also considered recently in \cite{Lima} and \cite{Sokal}, where it appeared in connection with the theory of branched continued fractions. In the former, explicit formulas for the type II polynomials on the step-line were found and some properties of the zeros were described.
\medbreak

Essential to our approach, compared to what was done before, will be the use of the Mellin transform. The Mellin transform of a function $f\in L^{1}(\R_{\geq 0})$ is given by
    $$\hat{f}(s)=\int_0^\infty f(x) x^{s-1} dx. $$ 
Two important properties of the Mellin transform are that there is an inverse transform
    $$f(x) = \frac{1}{2\pi i} \int_{c-i\infty}^{c+i\infty} \hat{f}(s) x^{-s} ds, $$
for appropriate $c\in\R$, and that the Mellin convolution of two functions $f,g\in L^{1}(\R_{\geq 0})$, which is given by 
    $$(f\ast g)(x) = \int_0^\infty f(t) g(x/t) \frac{dt}{t},$$
has Mellin transform
    $$\int_0^\infty (f\ast g)(x) x^{s-1} dx = \left(\int_0^\infty f(x) x^{s-1} dx\right) \cdot \left(\int_0^\infty g(x) x^{s-1} dx\right).$$  
Many results about the type I functions and type II polynomials can be obtained, very naturally, via the Mellin transform, and for a larger set of multi-indices than the usual step-line.
\medbreak

The multiple orthogonal polynomials in the Laguerre- and Jacobi-like setting (i.e. $p\geq q$) will appear in the study of the squared singular values of products of truncated Haar distributed unitary random matrices and Ginibre matrices (matrices with independent standard complex Gaussian random variables as its entries). Such products were for example already investigated in \cite{KieburgKuijlaarsStivigny,KuijlaarsStivigny}. The connection with multiple orthogonality was only established in the case where one considers a product of Ginibre matrices, see \cite{KuijlaarsZhang}, or the product of a Ginibre matrix with a truncated unitary random matrix, see \cite{VAWolfs}. In those cases, the associated weights were respectively Meijer-$G$ functions generalizing the $K$-Bessel functions and exponential integrals. The role of multiple orthogonality for other kinds of products has, to our knowledge, not been investigated before. Other families of multiple orthogonal polynomials also appear in random matrix theory, see \cite{Kuijlaars1,Kuijlaars2}.

The multiple orthogonal polynomials in the Bessel- and Jacobi-like setting (i.e. $p\leq q$) will appear as the solutions of certain Hermite-Padé approximation problems associated with hypergeometric series. Similar approximation problems were considered in \cite{Nesterenko} in connection with the theory of branched continued fractions, but the overall quality of the approximants was not studied. We will show that the Diophantine and approximation quality will be good enough to prove $\Q$-linear independence of the approximated hypergeometric series and in particular their irrationality. Such results are also implied by Siegel's method for $E$-functions (see, e.g., \cite{ShidlovskiiBook} for an introduction), which was further improved in \cite{Shidlovskii, Chudnovsky, Zudilin}. However, since Siegel's method is based on approximants that are only known implicitly, they raise the question of whether such results can also be proven by making use of approximants that can be constructed explicitly. Thus, by giving a proof using certain explicit Hermite-Padé approximants, we will provide an answer to this question. Hermite-Padé approximation was also used in Apéry's proof of the irrationality of $\zeta(3)$ in \cite{Apery} as explained in \cite{Beukers}.

\section{Jacobi-like setting}
\subsection{System of weights} \label{J_W}

Suppose that $\vec{a},\vec{b}\in (-1,\infty)^r$ satisfy $a_j<b_j$ for $1\leq j\leq r$. We will consider the $r$ weights $w_j(x;\vec{a},\vec{b}) = w_0(x;\vec{a},\vec{b}+\vec{e}_j)$, $1\leq j\leq r$, on $(0,1)$ (here $\vec{e}_j$ denotes the unit vector with a $1$ in the $j$-th component). In this setting, a weight $w_0(x;\vec{a},\vec{b})$ with moments $m_k(\vec{a},\vec{b})$ as in \eqref{Intro_mom} can be constructed by taking the Mellin convolution of the $r$ beta densities $\mathcal{B}^{a_j,b_j}$, $1\leq j\leq r$. Indeed, a beta density $\mathcal{B}^{a,b}$ has Mellin transform $\Gamma(s+a)/\Gamma(s+b)$. Since Carleman's condition $\sum_{k=0}^\infty m_k(\vec{a},\vec{b})^{-1/(2k)} = \infty $ is satisfied, there can be no other (positive) weights on the positive real line with those moments. The weight $w_j(x;\vec{a},\vec{b})$ then has Mellin transform 
$$ \int_0^1 x^{s-1} w_j(x;\vec{a},\vec{b}) dx = \frac{\Gamma(s+\vec{a})}{\Gamma(s+\vec{b}+\vec{e}_j)}$$
and arises as the Mellin convolution of $w_0(x;\vec{a},\vec{b})$ with the Jacobi-Piñeiro weight $x^{b_j}$. Note that the confluent case $\vec{a}=\vec{b}$ reduces to the Jacobi-Piñeiro setting in which we use the weights $x^{b_j}$ for $1\leq j\leq r$.
\medbreak

We will study the multiple orthogonal polynomials associated with the weights $(w_1,\dots,w_r)$ for multi-indices near the diagonal. This is a larger set of multi-indices than the usual step-line. The precise definition of these sets of multi-indices is stated below.

\begin{defi}
    A multi-index $\vec{n}\in\N^r$ is said to be near the diagonal if $\left|n_i-n_j\right| \leq 1$ for all $i,j$. On the other hand, we say that a multi-index $\vec{n}\in\N^r$ is on the step-line if $n_1\geq \dots\geq n_r\geq n_1-1$. The sets of all such multi-indices are denoted by $\mathcal{N}^r$ and $\mathcal{S}^r$ respectively.
\end{defi}

First, we will show that under certain conditions on the parameters the set of weights $\{w_j(x)\}_{j=1}^r$ is an AT-system for multi-indices near the diagonal. For this we need to prove that the (extended) set of weights $\{x^{k} w_j(x) \mid 0\leq k\leq n_j-1,\, 1\leq j\leq r\}$ is a T-system, i.e. that every non-zero combination $F(x) = \sum_{j=1}^r A_{j}(x) w_j(x)$, with $\deg A_{j} \leq n_j-1$, has at most $\sz{n}-1$ zeros in $(0,1)$. We will prove this by analyzing the Mellin transform of such polynomial combinations. As a consequence, multi-indices near the diagonal will be normal and some properties of the zeros of the corresponding multiple orthogonal polynomials will be known.

\begin{lem} \label{J_PC}
Let $\vec{n}\in\mathcal{N}^r$. The Mellin transform of $F(x) = \sum_{j=1}^r A_{j}(x) w_j(x)$, with $\deg A_{j} \leq n_j-1$, is of the form
    $$\frac{\Gamma(s+\vec{a})}{\Gamma(s+\vec{b}+\vec{n})} p_{\sz{n}-1}(s),$$ 
where $p_{\sz{n}-1}(s)$ is a polynomial of degree at most $\sz{n}-1$.
\end{lem}
\begin{proof}
The Mellin transform of $F(x)$ is given by
    $$ \int_0^1 F(x) x^{s-1} dx = \sum_{j=1}^r \sum_{k=0}^{n_j-1} A_{j}[k] \int_0^1 x^{k+s-1} w_j(x) dx, $$
where
    $$ \int_0^1 x^{k+s-1} w_j(x) dx = \frac{\Gamma(s+k+\vec{a})}{\Gamma(s+k+\vec{b}+\vec{e}_j)} = 
    \frac{\Gamma(s+\vec{a})}{\Gamma(s+\vec{b})} \frac{(s+\vec{a})_k}{(s+\vec{b})_k (s+b_j+k)}. $$
It then remains to show that 
$$\sum_{j=1}^r \sum_{k=0}^{n_j-1} A_{j}[k] \frac{(s+\vec{a})_k}{(s+\vec{b})_k (s+b_j+k)} = \frac{p_{\sz{n}-1}(s)}{\prod_{i=1}^r (s+b_i)_{n_i}}.$$
If we multiply the left-hand side by $\prod_{i=1}^r (s+b_i)_{n_i}$, we obtain
\begin{equation} \label{J_PC_poly}
    \sum_{j=1}^r \sum_{k=0}^{n_j-1} A_{j}[k] (s+\vec{a})_k \prod_{i=1}^r \frac{(s+b_i)_{n_i}}{(s+b_i)_{k+\delta_{i,j}}},
\end{equation}
which is a polynomial if $n_i\geq n_j-1+\delta_{i,j}$ for all $i,j$ (here $\delta_{i,j}$ denotes the Kronecker delta). The latter leads precisely to the near diagonal condition on $\vec{n}$ that was assumed. The polynomial $p_{\sz{n}-1}(s)$ is of degree at most $\sz{n}-1$, because, after comparing with the left-hand side, the ratio $p_{\sz{n}-1}(s)/\prod_{i=1}^r (s+b_i)_{n_i}$ needs to be $O(s^{-1})$ as $s\to\infty$.
\end{proof}

In what follows, we will use the fact that the underlying system of Jacobi-Piñeiro weights $(x^{b_1},\dots,x^{b_r})$ already forms an AT-system for every multi-index under the condition that $b_i-b_j\not\in\Z$ whenever $i\neq j$, see \cite[Section 23.3.2]{Ismail}. Given an appropriate kernel $K(x,y)$, a T-system $\{u_j(x)\}_{j=1}^n$ can be lifted to another T-system $\{v_j(y)\}_{j=1}^n$ defined by $v_j(y) = \int_\R u_j(x) K(x,y) dx$ by making use of the composition formula 
\begin{equation} \label{CompForm}
    \left| v_j(y_k) \right|_{j,k=1}^n = \int_{\vec{x}\in X^n(\uparrow)} |u_j(x_l)|_{j,l=1}^n |K(x_l,y_k)|_{l,k=1}^n dx_1\dots dx_n,
\end{equation}
where $X^n(\uparrow) = \{\vec{x}\in X^n \mid x_1<\dots<x_n\}$ (see \cite[Example 8 - Chapter 1]{KarlinStudden}; it is a special case of Andréief's integration formula \cite{Andreief}). Indeed, the first determinant on the right-hand side is non-zero and has a constant sign for every $\vec{x}\in X^n(\uparrow)$ because $\{u_j(x)\}_{j=1}^n$ is a T-system. Hence, if $|K(x_l,y_k)|_{l,k=1}^n$ doesn't change sign for all $\vec{x}\in X^n(\uparrow)$ and $\vec{y}\in Y^n(\uparrow)$ and $|K(x_k,y_l)|_{k,l=1}^n \neq 0 $ on a non-null subset of $X^n(\uparrow)$ for fixed $\vec{y}\in Y^n(\uparrow)$, the left-hand side shares the same property and the new system is also a T-system. Using this idea, we can prove the following extension to more general weights.

\begin{prop} \label{J_AT}
    Suppose that all $b_j-a_j\in\Z_{\geq 0} \cup (\sz{n}-1,\infty)$. Then the set of weights $\{w_j(x)\}_{j=1}^r$ is an AT-system on $(0,1)$ for every $\vec{n}\in\mathcal{N}^r$, if additionally $b_i-b_j\not\in\Z$ whenever $i\neq j$ or $b_J-a_J>\sz{n}-n_J $ for some $J$.
\end{prop}
\begin{proof}

Let $F(x) = \sum_{j=1}^r A_{j}(x) w_j(x)$ be a polynomial combination of the weights with $\deg A_{j} \leq n_j-1$. It follows from Lemma \ref{J_PC} that the Mellin transform of $F(x)$ is of the form 
\begin{equation} \label{J_PC_MT}
    \frac{\Gamma(s+\vec{a})}{\Gamma(s+\vec{b}+\vec{n})} p_{\sz{n}-1}(s).
\end{equation}
We will interpret this result in two ways, each leading to a different set of conditions on the parameters. 
\smallbreak

First, since \eqref{J_PC_MT} is the product of $p_{\sz{n}-1}(s)/\prod_{i=1}^r (s+b_i)_{n_i}$ and $\Gamma(s+\vec{a})/\Gamma(s+\vec{b})$, $F(x)$ is given by the Mellin convolution of a polynomial combination $G(x) = \sum_{j=1}^r B_{j}(x) x^{b_j}$ with $\deg B_j \leq n_j-1$ (see Lemma \ref{J_PC} with $\vec{a}=\vec{b}$) and $w_0(x)$ (which arose as the Mellin convolution of the beta densities $\mathcal{B}^{a_j,b_j}(x)$ for $1\leq j\leq r$). Hence, $F(x)$ can also be written as a linear combination of the weights in $\{x^{k+b_j} \ast w_0(x) \mid 0\leq k\leq n_j-1,\, 1\leq j\leq r\}$. Since we know that the set of weights $\{x^{k+b_j} \mid 0\leq k\leq n_j-1,\, 1\leq j\leq r\}$ is a T-system if $b_i-b_j\not\in\Z$ whenever $i\neq j$, it is sufficient to show that taking the consecutive Mellin convolution with the beta densities $\mathcal{B}^{a_j,b_j}(x)$ preserves this property. Convoluting with a beta density $\mathcal{B}^{a,b}(x)$ corresponds to taking $K(x,y)=(x/y)^{a}(1-x/y)_+^{b-a-1}/y$ on $(0,1)^2$ in the composition formula \eqref{CompForm}. It is known that $(y-x)_+^{b-a-1}$ is totally positive kernel of order $n=\sz{n}$ if $b-a-1\in \Z_{\geq 0} \cup (n-2,\infty)$, and thus in particular $\left|(y_j-x_i)_+^{b-a-1}\right|_{i,j=1}^{n} \geq 0$ for all $\vec{x}\in X^n(\uparrow)$ and $\vec{y}\in Y^n(\uparrow)$ (see \cite[Theorem 9.1]{Karlin} and \cite{KarlinStudden} for an introduction to totally positive kernels and related topics). On the other hand, it is straightforward to show that $\left|(y_j-x_i)_+^{b-a-1}\right|_{i,j=1}^n = (y_1-x_1)^{b-a-1} \dots (y_n-x_n)^{b-a-1} >0$ if $x_1 < y_1 \leq x_2 < y_2 \leq \dots \leq x_n < y_n$. Hence the resulting set of weights will again be a T-system.
\smallbreak

Another way to interpret the polynomial combination $F(x)$ is to start from a function with Mellin transform $\Gamma(s+a_J)/\Gamma(s+b_J+n_J) p_{\sz{n}-1}(s)$ and to take the convolution with the missing beta densities $\mathcal{B}^{a_j,b_j+n_j}(x)$. Such a function exists under the assumption that $b_J+n_J-a_J>\sz{n}$. Indeed, in that case we can do a partial fraction decomposition and take the inverse Mellin transform to show that it is given by a linear combination of the weights in $\{x^{k+a_J}(1-x)^{b_J+n_J-a_J-\sz{n}-1} \mid 0\leq k\leq \sz{n}-1\}$. The latter is a T-system as implied by the Jacobi(-Piñeiro) case and we can proceed similarly as before.
\end{proof}

The bottleneck of the above approach is that in order for $(y-x)_+^{b-a-1}$ to be a totally positive kernel of order $\sz{n}$, we have to assume that $b-a-1\in \Z_{\geq 0} \cup (\sz{n}-2,\infty)$. It is not possible to relax this condition as it was shown recently in \cite[Theorem 1.7]{Khare} that these two statements are equivalent (see \cite{B-G-K-P} for an extensive discussion on the appearance of this kind of condition).

\begin{cor} \label{J_AT_sys}
    Suppose that all $b_j-a_j\in\Z_{\geq 0}$ and $b_i-b_j\not\in\Z$ whenever $i\neq j$. Then the set of weights $\{w_j(x)\}_{j=1}^r$ is an AT-system on $(0,1)$ for every multi-index near the diagonal.
\end{cor}

It is known that the classical orthogonality weights satisfy a Pearson equation of the form
    $$[\sigma(x) w(x)]' = \tau(x) w(x),$$
where $\sigma,\tau$ are polynomials with $\deg \sigma \leq 2$ and $\deg \tau = 1$ (see, e.g., \cite[Chapter 23]{Ismail}). For the beta density $\mathcal{B}^{a,b}(x)$, we have $\sigma(x)=x(1-x)$ and $\tau(x)=(a+1)-(b+1)x$. Instead of a Pearson equation, the system of weights is now governed by a matrix Pearson equation.
    
\begin{prop} \label{J_W_MPE}
Suppose that $b_i\neq b_j$ whenever $i\neq j$. Define $\vec{W} = (w_1,\dots,w_r)^T$. Then
    $$[x(1-x)\vec{W}(x)]' = \mathcal{T}(x) \vec{W}(x),$$
where $\mathcal{T}(x)$ is a matrix with entries
     $$\mathcal{T}_{k,j}(x) = - \frac{\prod_{i=1}^r (a_i-b_j)}{\prod_{i\neq j}^r (b_i-b_j)} + \begin{dcases}
            0,\ k\neq j, \\
            x - (b_j+1)(1-x),\ k=j.
        \end{dcases}$$
\end{prop}
\begin{proof}
We aim to find a relation between the Mellin transforms $\hat{w}_J(s+1)$ and $\hat{w}_j(s)$ for $1\leq j\leq r$. 
Note that 
        $$\sum_{j=1}^r [c_j+s\delta_{j,J}] \hat{w}_j(s) = \frac{\Gamma(s+\vec{a})}{\Gamma(s+\vec{b}+1)} \sum_{j=1}^r [c_j + s \delta_{j,J}] \prod_{i\neq j} (s+b_i) $$
and that the sum on the right-hand side is a monic polynomial of degree $r$ with $r$ parameters $c_j$. Hence, we can try to pick the parameters in such a way that the polynomial becomes $\prod_{i=1}^r (s+a_i)$, which would then lead to the relation
    \begin{equation} \label{J_W_MPE_rel}
        \sum_{j=1}^r [c_j+s\delta_{j,J}] \hat{w}_j(s) = (s+b_J+1) \hat{w}_J(s+1).
    \end{equation}
It is indeed possible to do this. Taking $s=-b_k$ with $k\neq J$, we find that
    $c_k \prod_{i\neq k} (b_i-b_k) = \prod_{i=1}^r (a_i-b_k)$, 
while taking $s=-b_J$ gives
        $(c_J-b_J) \prod_{i\neq J} (b_i-b_J) = \prod_{i=1}^r (a_i-b_J)$.
Hence, 
        $$c_j = \frac{\prod_{i=1}^r (a_i-b_j)}{\prod_{i\neq j}^r (b_i-b_j)} + \begin{dcases}
            0,\ j\neq J, \\
            b_J,\ j=J.
        \end{dcases}$$
Formula \eqref{J_W_MPE_rel} can be written as
        $$(s-1)\hat{w}_J(s+1)-(s-1)\hat{w}_J(s) = \sum_{j=1}^r [c_j+\delta_{j,J}] \hat{w}_j(s) - (b_J+2) \hat{w}_J(s+1),$$
after which it is straightforward to take the inverse Mellin transform and obtain
        $$(x^2 w_J(x))' - (x w_J(x))' = \sum_{j=1}^r [c_j+\delta_{j,J}] w_j(x) - (b_J+2) x w_J(x).$$   
The fact that $\int_0^1 (x^2 w_J(x))' x^{s-1} dx = (s-1)\hat{w}_J(s+1)$ and $\int_0^1 (x w_J(x))' x^{s-1} dx = (s-1)\hat{w}_J(s)$ follows from integration by parts and the property $w_J(1)=0$. Combining these relations for $1\leq J\leq r$ then leads to the desired result.
\end{proof}

\subsection{Type I polynomials}

In what follows, we will study the type I multiple orthogonal polynomials associated with the weights $(w_1,\dots,w_r)$ introduced in the previous section for multi-indices near the diagonal. Essential to our approach will be the fact that we can find an explicit expression for the Mellin transform of the associated type I functions.

\begin{thm} \label{J_MT}
Let $\vec{n}\in\mathcal{N}^r$. Every type I function $F_{\vec{n}}$ has a Mellin transform of the form
    $$\mathcal{F}_{\vec{n}} \cdot \frac{\Gamma(s+\vec{a})}{\Gamma(s+\vec{b}+\vec{n})} (1-s)_{\sz{n}-1},\quad \mathcal{F}_{\vec{n}}\in\R.$$
\end{thm}
\begin{proof}      
It follows from Lemma \ref{J_PC} that the Mellin transform of $F_{\vec{n}}$ is of the form 
    $$\frac{\Gamma(s+\vec{a})}{\Gamma(s+\vec{b}+\vec{n})} p_{\sz{n}-1}(s),$$
where $p_{\sz{n}-1}(s)$ is a polynomial of degree at most $\sz{n}-1$. The orthogonality conditions then imply that this polynomial must be given by a scalar multiple of $(1-s)_{\sz{n}-1}$.
\end{proof}

It turns out that if $a_j=b_j$ for some $j$, we can be more generous with the restrictions on the multi-index, for example in the Jacobi-Piñeiro setting (where $\vec{a}=\vec{b}$) the above holds for all multi-indices. However, we won't go into more detail on this phenomenon here.
\medbreak

From now on, we will assume that the (non-zero) type I functions $F_{\vec{n}}(x)$ are normalized such that $\mathcal{F}_{\vec{n}} = 1$.
\medbreak

Theorem \ref{J_MT} allows us to prove normality of multi-indices near the diagonal under a different set of conditions on the parameters than in Corollary \ref{J_AT_sys}. These conditions are less restrictive. Indeed, if $b_j-a_i\in\Z_{<0}$ for some $i$ and $j$, then we must also have $b_j-b_i\in\Z_{<0}$ whenever $b_i-a_i\in\Z_{\geq 0}$.

\begin{prop}
    Suppose that $b_j-a_i\not\in\Z_{<0}$ and $b_i-b_j\not\in\Z$ whenever $i\neq j$. Then every multi-index near the diagonal is normal for the set of weights $\{w_j(x)\}_{j=1}^r$.
\end{prop}
\begin{proof}
Under the stated conditions, the leading coefficients of the type I polynomials can be determined explicitly as the unique residues at $s=-b_j-n_j+1$. Indeed, one has,
    $$A_{\vec{n},j}[n_j-1] \frac{(\vec{a}-b_j-n_j+1)_{n_j-1}}{(\vec{b}-b_j-n_j+1)_{n_j-1}} = \frac{(-1)^{n_j-1}}{(n_j-1)! \prod_{i=1,i\neq j}^r (b_i-b_j-n_j+1)_{n_i} } (b_j+n_j)_{\sz{n}-1}, $$
from which it follows that $A_{\vec{n},j}[n_j-1]\neq 0$.
\end{proof}

The inverse Mellin transform immediately provides a representation as a contour integral for the type I functions.

\begin{cor} \label{J_IR}
    Let $\vec{n}\in\mathcal{N}^r$. Then, 
    $$F_{\vec{n}}(x) = \frac{1}{2\pi i} \int_{c-i\infty}^{c+i\infty} \frac{\Gamma(s+\vec{a})}{\Gamma(s+\vec{b}+\vec{n})} (1-s)_{\sz{n}-1} x^{-s} ds, $$
    for some $c>0$ with $-\min\{a_j\mid j\} < c < 1$.
\end{cor}

A Rodrigues-type formula also follows from Theorem \ref{J_MT}.
\begin{cor}
Let $\vec{n}\in\mathcal{N}^r$. Then,
    $$F_{\vec{n}}(x) = \frac{d^{\sz{n}-1}}{dx^{\sz{n}-1}}\left[ x^{\sz{n}-1} w_0(x;\vec{a},\vec{b}+\vec{n}) \right].$$
\end{cor}
\begin{proof}
    We can show this by taking the Mellin transform and by applying integration by parts $\sz{n}-1$ times. Since $w_0(x;\vec{a},\vec{b}+\vec{n})=O((1-x)^{\sz{n}})$ as $x\stackrel{<}{\to} 1$, we then obtain the desired result.
\end{proof}

The type I functions can be interpreted as certain Mellin convolutions as well.

\begin{prop} \label{J_MC}
Let $\vec{n}\in\mathcal{N}^r$. Then,
    $$F_{\vec{n}} = P_{\vec{n}}^{(I\mid \vec{b})} \ast \mathcal{B}^{a_1,b_1} \ast \dots \ast \mathcal{B}^{a_r,b_r}, $$
in terms of the type I function $P_{\vec{n}}^{(I\mid \vec{b})}$ associated with Jacobi-Piñeiro weights $(x^{b_1},\dots,x^{b_r})$.
\end{prop}
\begin{proof}
The Mellin transform of the beta density $\mathcal{B}^{a_j,b_j}$ is $\Gamma(s+a_j)/\Gamma(s+b_j)$, while the type I function $P_{\vec{n}}^{(I\mid \vec{b})}$ in the Jacobi-Piñeiro setting has Mellin transform $ (1-s)_{\sz{n}-1}/\prod_{j=1}^{q} (s+b_j)_{n_j}$ (set $\vec{a}=\vec{b}$ in Theorem \ref{J_MT}, or see \cite{SmetVA}). Hence the result follows from the multiplicative property of the Mellin convolution and the uniqueness of the Mellin transform.
\end{proof}

We can also express the type I function as a combination of certain hypergeometric series. Here we will use the following notation: given a vector $\vec{a}\in\R^r$, we denote $\vec{a}^{\ast j}$ for the vector $(a_1,\dots,a_{j-1},a_{j+1},\dots,a_r)\in\R^{r-1}$.

\begin{prop} \label{J_HS}
    Let $\vec{n}\in\mathcal{N}^r$ and suppose that $a_i-a_j\not\in\Z$ whenever $i\neq j$ and $b_i-a_j\not\in\Z$ for all $i,j$. Then,
        $$F_{\vec{n}}(x) = \sum_{j=1}^r \frac{(a_j+1)_{\sz{n}-1} \Gamma(\vec{a}^{\ast j}-a_j)}{\Gamma(\vec{b}-a_j+\vec{n})} {}_{r+1}F_r \left( \begin{array}{c} \sz{n}+a_j, -\vec{n}-\vec{b}+a_j+1 \\ a_j+1,-\vec{a}^{\ast j}+a_j+1 \end{array}; x \right) x^{a_j}. $$
\end{prop}
\begin{proof}
    We will work out the contour integral in Corollary \ref{J_IR} via the residue theorem. The integrand has simple poles at $s=-a_j-k$ for $k\in\Z_{\geq 0}$ and $1\leq j\leq r$ with residue
    $$\begin{aligned}
        x^{a_j+k} \frac{(-1)^k}{k!} & \frac{\Gamma(\vec{a}^{\ast j}-a_j-k)}{\Gamma(\vec{b}-a_j+\vec{n}-k)} (a_j+k+1)_{\sz{n}-1} \\
        &= x^{a_j+k} \frac{(-1)^k}{k!} \frac{\Gamma(\vec{a}^{\ast j}-a_j)}{(\vec{a}^{\ast j}-a_j-k)_k} \frac{(\vec{b}-a_j+\vec{n}-k)_k}{\Gamma(\vec{b}-a_j+\vec{n})} \frac{\Gamma(a_j+k+\sz{n})}{\Gamma(a_j+k+1)}.
        \end{aligned}$$
By making use of the identity $(x)_n=(-1)^n (-x-n+1)_n$, we can rewrite the latter as
     $$   x^{a_j+k} \frac{1}{k!} \frac{\Gamma(\vec{a}^{\ast j}-a_j)}{(a_j-\vec{a}^{\ast j}+1)_k} \frac{(a_j-\vec{b}-\vec{n}+1)_k}{\Gamma(\vec{b}-a_j+\vec{n})} \frac{(a_j+\sz{n})_k}{(a_j+1)_k} \frac{\Gamma(a_j+\sz{n})}{\Gamma(a_j+1)},
    $$
which then leads to the desired result.
\end{proof}

Observe that the terms that appear in this sum (so including the $x^{a_j}$) are all solutions to the standard hypergeometric differential equation, see \cite[Eq. 16.8.6]{DLMF}, for
\begin{equation} \label{J_DE_sol}
    {}_{r+1}F_r \left( \begin{array}{c} \sz{n}, -\vec{n}-\vec{b}+1 \\ -\vec{a}+1 \end{array}; x \right) .
\end{equation}
Hence the type I function itself satisfies the same differential equation.
\medbreak

Finally, we will derive some expressions for the type I polynomials.

\begin{thm} \label{J_I_poly}
    Let $\vec{n}\in\mathcal{N}^r$ and suppose that $b_j-a_i\not\in\Z_{<0}$ and $b_i-b_j\not\in\Z_{\geq 0}$ whenever $i\neq j$. The type I polynomials $A_{\vec{n},j}$ are given by
    $$ A_{\vec{n},j}(x) = \frac{(\vec{a}-b_j)_1}{(\vec{b}^{\ast j}-b_j)_1} \sum_{J=1}^r \sum_{K=0}^{n_J-1} P_{\vec{n},J}[K] \sum_{k=0}^{K-1+\delta_{j,J}} \frac{(\vec{b}^{\ast j}-b_J-K)_{k+1} (b_j-b_J-K)_k}{(\vec{a}-b_J-K)_{k+1}} x^k, $$
    where
    $$ P_{\vec{n},J}[K] = \frac{(b_J+K+1)_{\sz{n}-1}}{\prod_{i=1,i\neq J}^r (b_i-b_J-K)_{n_i}} \frac{(-1)^K}{K!(n_J-K-1)!} .$$
\end{thm}
\begin{proof}
We will make use of the connection between multiple orthogonal polynomials and Hermite-Padé approximation (see, e.g., \cite[Chapter 23]{Ismail}). It is known that there exists a polynomial $B_{\vec{n}}(z)$ such that
    $$\sum_{j=1}^r A_{\vec{n},j}(z) f_j(z) - B_{\vec{n}}(z) =  \sum_{k\geq 0} \hat{F}_{\vec{n}}(k+1) \frac{1}{z^{k+1}}, \quad f_j(z) = \sum_{k\geq 0} \frac{\Gamma(k+\vec{a}+1)}{\Gamma(k+\vec{b}+\vec{e}_j+1)} \frac{1}{z^{k+1}}.$$
We can obtain an expression for the $A_{\vec{n},j}(z)$ by working out the right-hand side of the above. It follows from Theorem \ref{J_MT} that 
$$\hat{F}_{\vec{n}}(s) = \frac{\Gamma(s+\vec{a})}{\Gamma(s+\vec{b})} \frac{(1-s)_{\sz{n}-1}}{\prod_{j=1}^q (s+b_j)_{n_j}}. $$
A partial fraction decomposition can be used to expand the second factor as
    $$\frac{(1-s)_{\sz{n}-1}}{\prod_{j=1}^r (s+b_j)_{n_j}} = \sum_{j=1}^r \sum_{l=0}^{n_j-1} \frac{P_{\vec{n},j}[l]}{s+l+b_j}, $$
in terms of the coefficients $P_{\vec{n},j}[l]$ of the type~I Jacobi-Piñeiro polynomials associated with the weights $(x^{b_1},\dots,x^{b_r})$. We then obtain
\begin{equation} \label{J_I_poly_EE}
    \sum_{k\geq 0} \hat{F}_{\vec{n}}(k+1) \frac{1}{z^{k+1}} =  \sum_{j=1}^r \sum_{l=0}^{n_j-1} P_{\vec{n},j}[l] \sum_{k\geq 0} \frac{\Gamma(k+\vec{a}+1)}{\Gamma(k+\vec{b}+1)} \frac{1}{k+l+b_j+1} \frac{1}{z^{k+1}},
\end{equation}
and it remains to find an operation to lower $L$ in 
    $$ I_{J,L}(z) = \sum_{k\geq 0} \frac{\Gamma(k+\vec{a}+1)}{\Gamma(k+\vec{b}+1)} \frac{1}{k+L+b_J+1} \frac{1}{z^{k+1}},$$
because $I_{J,0}(z) = f_J(z)$. We can take the first term out of the series to get
    $$I_{J,L-1}(z) = \frac{\Gamma(\vec{a}+1)}{\Gamma(\vec{b}+1)} \frac{1}{L+b_J} + \sum_{k\geq 0} \frac{\Gamma(k+\vec{a}+1)}{\Gamma(k+\vec{b}+1)} \frac{(k+\vec{a}+1)_1}{(k+\vec{b}+1)_1} \frac{1}{k+L+b_J+1} \frac{1}{z^{k+2}}.$$
A partial fraction decomposition in the variable $k$ then gives
\begin{equation} \label{J_I_poly_PFD}
    \frac{(k+\vec{a}+1)_1}{(k+\vec{b}+1)_1} \frac{1}{k+L+b_J+1} = \sum_{j=1}^r \frac{\alpha_j^{L}}{k+b_j+1} + \frac{\alpha_0^{L}}{k+L+b_J+1},
\end{equation}
where
    $$\alpha_j^{L} = \frac{(\vec{a}-b_j)_1}{(\vec{b}^{\ast j}-b_j)_1}\frac{1}{L+b_J-b_j},\quad \alpha_0^{L} = \frac{(\vec{a}-b_J-L)_1}{(\vec{b}-b_J-L)_1}.$$
Therefore,
    $$I_{J,L-1}(z) = \frac{\Gamma(\vec{a}+1)}{\Gamma(\vec{b}+1)} \frac{1}{L+b_J} + \frac{1}{z} \sum_{j=1}^r \alpha_j^{L} f_j(z) + \frac{\alpha_0^{L}}{z} I_{J,L}(z),$$
which is a recurrence relation of the form $I_{J,L}(z) = C_L(z) + D_L(z) I_{J,L-1}(z)$ with
    $$C_L(z) = - \frac{1}{\alpha_0^{L}} \left( \sum_{j=1}^r \alpha_j^{L} f_j(z) + \frac{\Gamma(\vec{a}+1)}{\Gamma(\vec{b}+1)} \frac{z}{L+b_J} \right),\quad D_L(z) = \frac{z}{\alpha_0^{L}} . $$
We can obtain a formula for $I_{J,L}(z)$ by consecutive application of this recurrence relation
$$I_{J,L}(z) = \sum_{l=0}^{L-1} C_{L-l}(z) \prod_{j=1}^l D_{L-l+j}(z) + \prod_{j=1}^L D_j(z) f_J(z). $$
Note that $\prod_{j=1}^l D_{L-l+j}(z) = z^l (\vec{b}-b_J-L)_{l} / (\vec{a}-b_J-L)_{l}$ so that
$$\begin{aligned}
    I_{J,L}(z) &= - \sum_{j=1}^r \frac{(\vec{a}-b_j)_1}{(\vec{b}^{\ast j}-b_j)_1} f_j(z) \sum_{l=0}^{L-1} \frac{(\vec{b}-b_J-L)_{l+1}}{(\vec{a}-b_J-L)_{l+1}} \frac{z^l}{L-l+b_J-b_j} \\
    & - \frac{\Gamma(\vec{a}+1)}{\Gamma(\vec{b}+1)} \sum_{l=0}^{L-1} \frac{(\vec{b}-b_J-L)_{l+1}}{(\vec{a}-b_J-L)_{l+1}} \frac{z^{l+1}}{L-l+b_J} \\
    & + z^L \frac{(\vec{b}-b_J-L)_{L}}{(\vec{a}-b_J-L)_{L}} f_J(z).
\end{aligned}$$
We can write this as
\begin{align} \label{J_I_poly_IE}
I_{J,L}(z) &= \sum_{j=1}^r \frac{(\vec{a}-b_j)_1}{(\vec{b}^{\ast j}-b_j)_1} f_j(z) \sum_{l=0}^{L-1+\delta_{j,J}} \frac{(\vec{b}^{\ast j}-b_J-L)_{l+1} (b_j-b_J-L)_l}{(\vec{a}-b_J-L)_{l+1}} z^l\nonumber \\
    & - \frac{\Gamma(\vec{a}+1)}{\Gamma(\vec{b}+1)} \sum_{l=0}^{L-1} \frac{(\vec{b}-b_J-L)_{l+1}}{(\vec{a}-b_J-L)_{l+1}} \frac{1}{L-l+b_J} z^{l+1},
\end{align}
because the term we added to the sum is
$$ \frac{(\vec{a}-b_J)_1}{(\vec{b}^{\ast j}-b_J)_1} f_J(z) \frac{(\vec{b}^{\ast J}-b_J-L)_{L+1}(-L)_{L}}{(\vec{a}-b_J-L)_{L+1}} z^L = f_J(z) \frac{(\vec{b}^{\ast J}-b_J-L)_{L}(-L)_{L}}{(\vec{a}-b_J-L)_{L}} z^L . $$
The stated formulas for the type I polynomials then follow.
\end{proof}

The polynomial of degree $K-1+\delta_{j,J}$ that appears together with the coefficient $P_{\vec{n},J}[K]$ in the above formula can be represented as the hypergeometric polynomial
    $$ \frac{(\vec{b}-b_J-K)_1^{\ast j}}{(\vec{a}-b_J-K)_1} \,{}_{r+1}F_r \left( \begin{array}{c} 1,\vec{b}-b_J+1-\vec{e}_j-K \\ \vec{a}-b_J+1-K \end{array}; x \right).$$
The type I polynomials can therefore be constructed with the same coefficients as in the Jacobi-Piñeiro setting, but using these ${}_{r+1}F_r$-hypergeometric polynomials as a basis instead of the monomials.

\subsection{Type II polynomials}

In this section, we will study the type II multiple orthogonal polynomials associated with the weights $(w_1,\dots,w_r)$ introduced in Section \ref{J_W} for multi-indices near the diagonal. Formulas for these type II polynomials were already obtained in \cite{Lima}, however, their proof is implicit in the sense that one simply verifies that the stated polynomials satisfy the desired orthogonality conditions by making use of a property of the underlying hypergeometric series. We will provide a constructive method to derive those formulas. Once an expression is known, a similar approach would also allow us to verify the orthogonality conditions directly. Additionally, our formulas are proven for a larger set of multi-indices: near the diagonal $\mathcal{N}^r$ instead of on the step-line $\mathcal{S}^r$. Afterwards, we will describe the asymptotic zero distribution and interpret it through free probability.

\begin{thm} \label{J_II_HS}
    Let $\vec{n}\in\mathcal{N}^r$. Every type II polynomial $P_{\vec{n}}(x)$ is of the form
    $$\mathcal{P}_{\vec{n}} \cdot {}_{r+1}F_r \left( \begin{array}{c} -\sz{n}, \vec{b}+\vec{n}+1 \\ \vec{a}+1 \end{array}; x \right),\quad \mathcal{P}_{\vec{n}}\in\R.$$
\end{thm}
\begin{proof}
Moving into the complex plane allows us to write
    $$ P_{\vec{n}}(x) =  \frac{1}{2\pi i} \sum_{k=0}^{\sz{n}} P_{\vec{n}}[k] \frac{\Gamma(k+\vec{a}+1)}{\Gamma(k+\vec{b}+1)} \int_{\Sigma} \frac{\Gamma(t+\vec{b}+1)}{\Gamma(t+\vec{a}+1)} x^t \frac{dt}{t-k},$$
in terms of a contour $\Sigma$ encircling the interval $[0,\sz{n}]$ avoiding the poles of the integrand that are not in this interval. Hence, we find
    $$ P_{\vec{n}}(x) =  \frac{1}{2\pi i} \int_{\Sigma} \frac{p_{\sz{n}}(t)}{(-t)_{\sz{n}+1}} \frac{\Gamma(t+\vec{b}+1)}{\Gamma(t+\vec{a}+1)} x^t dt,$$
where $p_{\sz{n}}(t)$ is a polynomial of degree at most $\sz{n}$ that is determined by the relation
    $$\frac{p_{\sz{n}}(t)}{(-t)_{\sz{n}+1}} = \sum_{k=0}^{\sz{n}} \frac{\Gamma(k+\vec{a}+1)}{\Gamma(k+\vec{b}+1)} \frac{P_{\vec{n}}[k]}{t-k}. $$
Let $s\in\N$. Then,
\begin{equation} \label{J_II_HS_CI}
    \int_0^1 P_{\vec{n}}(x) x^s w_j(x) dx = \frac{1}{2\pi i} \int_{\Sigma} \frac{p_{\sz{n}}(t)}{(-t)_{\sz{n}+1}} \frac{(t+\vec{a}+1)_s}{(t+\vec{b}+1)_s (s+t+b_j+1)} dt.
\end{equation}
Note that the integrand is a rational function of order $\BO(t^{-2})$ as $t\to\infty$, so that we can compute the integral by moving the contour away to infinity and by adding the residues of poles in $\C\backslash[0,\sz{n}]$. Since the integral needs to vanish for $0\leq s\leq n_j-1$ and $1\leq j\leq r$, we don't want to have any poles in that domain. Under the near-diagonal restriction, this implies that $\prod_{j=1}^r (t+b_j+1)_{n_j}$ needs to divide $p_{\sz{n}}(t)$. Since the former is a polynomial of degree $\sz{n}$, we must have $p_{\sz{n}}(t) = \mathcal{P}_{\vec{n}} \cdot \prod_{j=1}^r (t+b_j+1)_{n_j}$ for some $\mathcal{P}_{\vec{n}}\in\R$. Consequently, after comparing residues, we find
    $$P_{\vec{n}}[k] = \mathcal{P}_{\vec{n}} \cdot \frac{\Gamma(k+\vec{b}+1)}{\Gamma(k+\vec{a}+1)} \frac{\prod_{j=1}^r (k+b_j+1)_{n_j}}{(-k)_{k}(1)_{\sz{n}-k}} $$
and this leads to the desired result.
\end{proof}

Similarly as with the type I functions, if $a_j=b_j$ for some $j$, we can be more generous with the restrictions on the multi-index. Also, observe the similarities with the hypergeometric series in \eqref{J_DE_sol}.
\medbreak

In what follows, we will assume that the (non-zero) type II polynomials $P_{\vec{n}}(x)$ are normalized such that $\mathcal{P}_{\vec{n}} = \Gamma(\vec{b}+\vec{n}+1)/[\Gamma(\vec{a}+1) (\sz{n})!]$.
\medbreak

From the proof of Theorem \ref{J_II_HS}, it is clear that we have the following contour integral representation for the type II polynomials.

\begin{cor} \label{J_II_IR}
    Let $\vec{n}\in\mathcal{N}^r$. Then,
    $$P_{\vec{n}}(x) = \frac{1}{2\pi i} \int_{\Sigma} \frac{1}{(-t)_{\sz{n}+1}} \frac{\Gamma(t+\vec{b}+\vec{n}+1)}{\Gamma(t+\vec{a}+1)} x^t dt, $$
    where $\Sigma$ is a positively oriented contour that starts and ends at $\infty$ and encircles the positive real line, without enclosing any poles of the integrand not in $[0,\infty)$.
\end{cor}   

Note that there is some kind of duality between the contour integral representations for the type I functions in Corollary \ref{J_IR} and type II polynomials in Corollary \ref{J_II_IR}. Such duality is expected due to the underlying biorthogonality relations.
\medbreak

We will end this section by discussing some properties of the zeros of the type II polynomials that were obtained in \cite{Lima}. By exploiting the connection with branched continued fractions, one was able to prove the theorem below. Note that under the, generally more restrictive, condition that all $b_j-a_j\in\Z_{\geq 0} \cup (\sz{n}-1,\infty)$ and $b_J-a_J>\sz{n}-n_J $ for some $J$, it also follows from the fact that the underlying system of weights is an AT-system (see Proposition \ref{J_AT}).

\begin{thm}[Lima \cite{Lima}, Theorem 4.8] \label{J_RSZ}
    Let $\vec{n}\in\mathcal{S}^r$ and suppose that $b_j>a_i$ if $i\leq j$ and $b_j>a_i-1$ if $i>j$. The type~II polynomials $P_{\vec{n}}(x)$ have real, positive, simple zeros.
\end{thm}

Using an underlying recurrence relation, it was then shown that the Stieltjes transform $S(z)$ of the asymptotic zero distribution of $P_{\vec{n}}(x)$ with $\vec{n}\in\mathcal{S}^r$ as $\sz{n}\to\infty$ satisfies the algebraic equation $(S(z)+r^r/(r+1)^{r+1})^{r+1} = z S(z)^r$. From this, one could conclude that $P_{\vec{n}}(x)$ has the same asymptotic zero distribution as the type II Jacobi-Piñeiro polynomials; it was already proven in \cite{NeuschelVA} that the Stieltjes transform of the latter satisfies the same algebraic equation.

\begin{thm} [Lima \cite{Lima}, Theorem 8.6] \label{J_AZD}
    Suppose that $b_j>a_i$ if $i\leq j$ and $b_j>a_i-1$ if $i>j$. The asymptotic zero distribution of $P_{\vec{n}}(x)$ with $\vec{n}\in\mathcal{S}^r$ as $\sz{n}\to\infty$ has a density, with support on $(0,1)$, given by
    $$v_r(x) = \frac{r+1}{\pi x} \frac{\sin{\p}\sin{r\p}\sin{(r+1)\p}}{(r+1)^2(\sin{r\p})^2-2r(r+1)\sin{(r+1)\p\sin{r\p}\cos{\p}}+r^2(\sin{(r+1)\p})^2}, $$
    after the change of variables
    $$x=\frac{r^r}{(r+1)^{r+1}} \frac{(\sin{(r+1)\p})^{r+1}}{\sin\p(\sin{r\p})^r},\quad \p\in \left(0,\frac{\pi}{r+1}\right). $$
\end{thm}

Another way to interpret these results is through the (finite) free multiplicative convolution. The idea to use this tool to study the zeros of hypergeometric polynomials was developed recently in \cite{MF-M-P}. Many known (and new) results were collected here, some go back to Szeg\H{o} \cite{SzegoConv}. Recently, there has been renewed interest in the finite free convolution of polynomials due to the connection with free probability, see \cite{M-S-S}. Since the type~II polynomials here are hypergeometric polynomials that don't belong to the class that was studied in \cite{MF-M-P}, we can't use their results on the zero behavior and asymptotic zero distribution directly.
\medbreak

Following the definition in \cite{M-S-S}, the finite free multiplicative convolution (called $d$-th symmetric multiplicative convolution here) of two polynomials
    $$p(x) = \sum_{k=0}^{d} x^{d-k} (-1)^k p[k],\quad q(x) = \sum_{k=0}^{d} x^{d-k} (-1)^k q[k], $$
of degree at most $d$, is defined as
    $$ p(x)\boxtimes_d q(x) = \sum_{k=0}^{d} x^{d-k} (-1)^k p[k] q[k].  $$
An important property is that, given two sequences of polynomials $(p_d)_{d\in\N}$ and $(q_d)_{d\in\N}$ with $\deg p_d = \deg q_d =d$, the asymptotic zero distribution of the finite free multiplicative convolution $p_d\boxtimes_d q_d$ as $d\to\infty$ is given by the free multiplicative convolution $\mu\boxtimes \nu$ of the asymptotic zero distribution $\mu$ of $p_d$ and $\nu$ of $q_d$ as $d\to\infty$ (see \cite[Theorem 1.4]{A-GV-P}). The multiplicative convolution of two measures is defined through the $S$-transform as 
$$S_{\mu\boxtimes\nu}(z) = S_{\mu}(z) \cdot S_{\nu}(z).$$
The $S$-transform of a measure $\mu$ is defined by 
    $$S_{\mu}(z) = m_{\mu}^{-1}(z) \frac{z+1}{z},$$
where $m_{\mu}^{-1}(z)$ is the (formal) inverse of the formal power series $m_{\mu}(z)=\sum_{k\geq 1} m_{\mu,k} z^k$ with $m_{\mu,k}=\int_\R x^k d\mu(x)$. The condition $m_{\mu,1} \neq 0$ ensures that this inverse exists. Note that $S_{\mu}(z)$ is both the standard notation for the Stieltjes transform and the $S$-transform of $\mu$, which may cause some conflict of notation in what follows. To avoid any confusion, we will always denote Stieltjes transforms by $S(z)$, dropping the dependence on the measure $\mu$, and denote $S$-transforms in another way.
\medbreak

Unfortunately, in general, a direct decomposition of the type II polynomials $P_{\vec{n}}(x;\vec{a},\vec{b})$ isn't helpful to study its zeros. It would lead to
$$\begin{aligned}
    {}_{r+1}F_r & \left( \begin{array}{c} -\sz{n}, \vec{b}+\vec{n}+1 \\ \vec{a}+1 \end{array}; x \right) \\ 
    &= {}_{2}F_1 \left( \begin{array}{c} -\sz{n}, b_1+n_1+1 \\ a_1+1 \end{array}; x \right) \boxtimes_{\sz{n}} \cdots \boxtimes_{\sz{n}}  {}_{2}F_1 \left( \begin{array}{c} -\sz{n}, b_r+n_r+1 \\ a_r+1 \end{array}; x \right),
\end{aligned}$$
but the polynomials that appear here are Jacobi polynomials $P_{\sz{n}}^{(a_j,b_j-a_j+n_j-\sz{n})}(x)$ with a parameter that is typically negative and not an integer, thus having roots in the complex plane. Under the condition that all $b_j-a_j\in\Z_{\geq 0} \cup (\sz{n}-1,\infty)$ and $b_J-a_J>\sz{n}-n_J $ for some $J$, the Jacobi polynomials do have real, positive roots, and at least one of them also has simple roots. It is known that if the roots of a collection of polynomials are real, positive and at least one polynomial also has simple roots, the same property holds for their convolution (see \cite{SzegoConv} and \cite[Theorem 6]{K-S-V}). Hence under these conditions, we would be able to obtain the result in Theorem \ref{J_RSZ}. Remarkably, these were exactly the conditions that were required for the system of weights to be an AT-system for the multi-index $\vec{n}$ (see Proposition \ref{J_AT}). 
\medbreak

By making use of the fact that the type II Jacobi-Piñeiro polynomials have the same asymptotic zero distribution as the type II polynomials $P_{\vec{n}}(x;\vec{a},\vec{b})$, it is still possible to give a description of the asymptotic zero distribution using the free convolution.

\begin{thm} \label{J_AZDConv}
    Suppose that $b_j>a_i$ if $i\leq j$ and $b_j>a_i-1$ if $i>j$. The asymptotic zero distribution of $P_{\vec{n}}(x)$ for $\vec{n}\in\mathcal{S}^r$ as $\sz{n}\to\infty$ is given by $((1-1/r)\delta_1+(1/r)\mu_{r-1})^{\boxtimes r}$ in terms of the deformed arcsin distribution $\mu_{r-1}$ from Lemma \ref{DefArcsin}.
\end{thm}
\begin{proof}
It is sufficient to prove this for the type II Jacobi-Piñeiro polynomials. They are given by 
$$\begin{aligned}
    {}_{r+1}F_r & \left( \begin{array}{c} -\sz{n}, \vec{b}+\vec{n}+1 \\ \vec{b}+1 \end{array}; x \right) \\ 
    &=  {}_{2}F_1 \left( \begin{array}{c} -\sz{n}, b_1+n_1+1 \\ b_1+1 \end{array}; x \right) \boxtimes_{\sz{n}} \cdots \boxtimes_{\sz{n}}  {}_{2}F_1 \left( \begin{array}{c} -\sz{n}, b_r+n_r+1 \\ b_r+1 \end{array}; x \right),
\end{aligned}$$
where
    $${}_{2}F_1 \left( \begin{array}{c} -\sz{n}, b_j+n_j+1 \\ b_j+1 \end{array}; x \right) = (1-x)^{\sz{n}-n_j} P_{n_j}^{(b_j,\sz{n}-n_j)}(x) $$
and $P_{n_j}^{(b_j,\sz{n}-n_j)}(x)$ is a Jacobi polynomial on $[0,1]$. The asymptotic zero distribution of $(1-x)^{\sz{n}-n_j}$ is $\delta_1$, while the asymptotic zero distribution of $P_{n_j}^{(b_j,\sz{n}-n_j)}(x)$ is given by the measure $\mu_{r-1}$ in Lemma \ref{DefArcsin}.
\end{proof}

The study of the asymptotic zero distribution of Jacobi polynomials $P_n^{(\alpha_n,\beta_n)}$ with varying parameters $\alpha_n,\beta_n\in\R$ such that $\alpha_n\to\alpha\in\R$ and $\beta_n\to\beta\in\R$ has been carried out in \cite{MF-MG-O} for many cases. In the lemma below, we establish that their result for $\alpha,\beta>0$ remains true if $\alpha\beta=0$.

\begin{lem} \label{DefArcsin}
    The asymptotic zero distribution $\mu_\alpha$ of the Jacobi polynomial $P_{n}^{(a,\alpha_n)}(x)$ with $a>-1$, $\alpha_n>0$ for all $n$ and $\alpha_n/n\to \alpha$ as $n\to\infty$ has density
        $$u_\alpha(x) = \frac{\sqrt{4(\alpha+1)x-(\alpha+2)^2x^2}}{2\pi x(1-x)}, \quad 0<x<4(\alpha+1)/(\alpha+2)^2.$$
\end{lem}
\begin{proof}
The zeros of $P_{n}^{(a,\alpha_n)}(x)$ are real, simple and contained in the interval $[0,1]$, hence we can use the usual method to transform the differential equation for $P_{n}^{(a,\alpha_n)}(z)$ to an algebraic equation for the Stieltjes transform $S(z)$ of its asymptotic zero distribution (see, e.g., \cite[Section 3.1.4]{LeursVA} or \cite[Section 3.4.2]{VAWolfs}). The differential equation is given by
    $$z(1-z)y'' + [a+1 - (a+\alpha_n+2)z] y' + n(n+a+\alpha_n+1) y = 0, $$
see, e.g., \cite[Eq. (4.21.1)]{Szego}. The associated algebraic equation is
    $$ z(1-z) S^2 - \alpha z S + \alpha + 1 = 0,$$
which we may solve explicitly for two solutions
    $$S_{\pm}(z) = \frac{\alpha z \pm ((\alpha+2)^2z^2-4(\alpha+1)z)^{1/2}}{2z(1-z)}. $$
Here we consider the square root with a branch cut along the positive real line (as $S(z)$ should be defined for $z\in\C\backslash [0,\infty)$). Since $z S_{-}(z)\to 1$ as $z\to\infty$, we have
    $$u_\alpha(x) = - \lim_{y\to 0^+} \frac{1}{\pi i} \text{Im}(S_{-}(x+iy)) = \frac{\sqrt{4(\alpha+1)x-(\alpha+2)^2x^2}}{2\pi x(1-x)}, $$
for $0<x<4(\alpha+1)/(\alpha+2)^2$.
\end{proof}

Another interesting decomposition of the type II polynomials $P_{\vec{n}}(x;\vec{a},\vec{b})$ is in terms of the type II Jacobi-Piñeiro polynomial $P_{\vec{n}}^{(II\mid\vec{b})}(x)$ associated with $(x^{b_1},\dots,x^{b_r})$,
$$P_{\vec{n}}(x;\vec{a},\vec{b}) = P_{\vec{n}}^{(II\mid\vec{b})}(x) \boxtimes_{\sz{n}}  {}_{2}F_1 \left( \begin{array}{c} -\sz{n}, b_1+1 \\ a_1+1 \end{array}; x \right) \boxtimes_{\sz{n}} \cdots \boxtimes_{\sz{n}}  {}_{2}F_1 \left( \begin{array}{c} -\sz{n}, b_r+1 \\ a_r+1 \end{array}; x \right),$$
because of its similarities to the Mellin convolution formula for the type I function in Proposition \ref{J_MC}.

\section{Laguerre-like setting}
\subsection{System of weights} \label{L_W}
Let $p>q$ and suppose that $\vec{a}\in(-1,\infty)^p$ and $\vec{b}\in(-1,\infty)^q$ satisfy $a_j < b_j$ for $1\leq j \leq q$. We will consider $p$ weights $\{w_j(x;\vec{a},\vec{b})\}_{j=1}^q\cup\{v_j(x;\vec{a},\vec{b})\}_{j=1}^{p-q}$ on $(0,\infty)$. The first weights are defined by $w_j(x;\vec{a},\vec{b})=w_0(x;\vec{a},\vec{b}+\vec{e}_j)$, similarly as before. In this setting, a weight $w_0(x;\vec{a},\vec{b})$ with moments $m_k(\vec{a},\vec{b})$ as in \eqref{Intro_mom} can be constructed by taking the Mellin convolution of $q$ beta densities $\mathcal{B}^{a_j,b_j}$, $1\leq j\leq q$, and $p-q$ gamma densities $\mathcal{G}^{a_j}$, $q+1\leq j\leq p$. Indeed, a beta density $\mathcal{B}^{a,b}$ has Mellin transform $\Gamma(s+a)/\Gamma(s+b)$ and a gamma density $\mathcal{G}^a$ has Mellin transform $\Gamma(s+a)$. Since Carleman's condition $\sum_{k=0}^\infty m_k(\vec{a},\vec{b})^{-1/(2k)} = \infty $ is satisfied, there can be no other (positive) measures on the positive real line with those moments. The weight $w_j(x;\vec{a},\vec{b})$ then has Mellin transform 
$$ \int_0^1 x^{s-1} w_j(x;\vec{a},\vec{b}) dx = \frac{\Gamma(s+\vec{a})}{\Gamma(s+\vec{b}+\vec{e}_j)}.$$
The remaining weights arise as $v_j(x;\vec{a},\vec{b}) = (x \cdot d/dx)^{j-1}[w_0(x;\vec{a},\vec{b})]$ and have Mellin transforms
$$\int_0^\infty x^{s-1} v_j(x;\vec{a},\vec{b}) dx =\frac{\Gamma(s+\vec{a})}{\Gamma(s+\vec{b})} (s-1)^{j-1}.$$

It is natural to set up the weights in this way, because it leads to desirable limiting relations between the sets of weights in all cases with $p\geq q$. We will describe those relations in the following lemma. Recall that given a vector $\vec{a}\in\R^r$, we denote $\vec{a}^{\ast j}$ for the vector $(a_1,\dots,a_{j-1},a_{j+1},\dots,a_r)\in\R^{r-1}$. Similar limiting relations will then also hold for the type I and type II polynomials. 

\begin{lem} \label{L_lim_rel}
    Given a set of weights $\{w_j(x;\vec{a},\vec{b})\}_{j=1}^q\cup\{v_j(x;\vec{a},\vec{b})\}_{j=1}^{p-q}$, we can construct a new set of weights $\{w_j(x;\vec{a},\vec{b}^{\ast q})\}_{j=1}^{q-1}\cup\{v_j(x;\vec{a},\vec{b}^{\ast q})\}_{j=1}^{p-q+1}$ via the limiting relations
    \begin{itemize}
        \item[i)] $\lim_{b_q\to\infty} \Gamma(b_q) w_j(x/b_q;\vec{a},\vec{b}) = w_j(x;\vec{a},\vec{b}^{\ast q})$ for $1\leq j\leq q-1$,
        \item[ii)] $\lim_{b_q\to\infty} \Gamma(b_q+1) (\lambda_0 w_q(x/b_q;\vec{a},\vec{b}) + \sum_{j=1}^{J} \lambda_j v_j(x/b_q;\vec{a},\vec{b})) = v_{J+1}(x;\vec{a},\vec{b}^{\ast q})$ with $\lambda_j = (-1)^{J-j} (b_q+1)^{J-j} $ for $0\leq J\leq p-q$.
    \end{itemize}
\end{lem}
\begin{proof}
It suffices to show that these relations hold for the underlying Mellin transforms. The interchange of limit and integral is justified through Lebesgue's dominated convergence theorem as the weight functions decay sufficiently fast as $x\to\infty$. The first part immediately follows from the definition of the weights. For ii), we note that the coefficients are such that
$$ \lambda_0 \hat{w}_q(k+1;\vec{a},\vec{b}) + \sum_{j=1}^{J} \lambda_j \hat{v}_j(k+1;\vec{a},\vec{b}) = \hat{w}_q(k+1;\vec{a},\vec{b}) k^J.  $$
Indeed, the left-hand side is of the form 
    $$ \frac{\Gamma(k+\vec{a}+1)}{\Gamma(k+\vec{b}+1)} \left(\frac{\lambda_0}{k+b_q+1} + \sum_{j=1}^{J} \lambda_j k^{j-1}\right) = \hat{w}_q(k+1;\vec{a},\vec{b}) p_{J}(k)$$
in terms of a polynomial $p_{J}(k)$ of degree $J$ with $J+1$ unknown coefficients, which we may solve for $p_{J}(k)=k^J$. 
\end{proof}

In particular, the above procedure describes a way to convert a set of Jacobi-like weights $\{w_j(x;\vec{a},\vec{b})\}_{j=1}^q$ on $[0,1]$ (so $\vec{a},\vec{b}\in\R^q$) to a set of Laguerre-like weights $\{w_j(x;\vec{a},\vec{b}^{\ast q})\}_{j=1}^q\cup\{v_1(x;\vec{a},\vec{b}^{\ast q})\}$ on $[0,\infty)$; we have to use the relations
    $$\lim_{b_q\to\infty} \Gamma(b_q) w_j(x/b_q;\vec{a},\vec{b}) = w_j(x;\vec{a},\vec{b}^{\ast q}), \quad
    \lim_{b_q\to\infty} \Gamma(b_q+1) w_q(x/b_q;\vec{a},\vec{b}) = v_1(x;\vec{a},\vec{b}^{\ast q}). $$
\medbreak

The limiting relations indicate that we should be able to obtain formulas for the multiple orthogonal polynomials for multi-indices of the form $\vec{n}\sqcup \vec{m}:=(n_1,\dots,n_q,m_1,\dots,m_{p-q})$ that are near the diagonal, for which the second part, associated with the weights $\{v_j\}_{j=1}^{p-q}$, is restricted to the step-line (because we require weights that correspond to previous components in their construction).
\medbreak

We can show that, under some conditions, the system of weights will be an AT-system for the described multi-indices. To prove this, we will again analyze the Mellin transform of polynomial combinations of the weights.

\begin{lem} \label{L_PC}
Let $\vec{n}\sqcup\vec{m}\in\mathcal{N}^p$ with $\vec{m}\in\mathcal{S}^{p-q}$. The Mellin transform of a polynomial combination $F(x) = \sum_{j=1}^q A_{j}(x) w_j(x) + \sum_{j=1}^{p-q} A_{q+j}(x) v_j(x) $ with $\deg A_{j} \leq n_j-1$ and $\deg A_{q+j} \leq m_j-1$ is of the form
    $$\frac{\Gamma(s+\vec{a})}{\Gamma(s+\vec{b}+\vec{n})} p_{\sz{n}+\sz{m}-1}(s),$$ 
where $p_{\sz{n}+\vec{m}|-1}(s)$ is a polynomial of degree at most $\sz{n}+\sz{m}-1$.
\end{lem}
\begin{proof}
The Mellin transform of $F(x)$ is given by
    $$ \int_0^1 F(x) x^{s-1} dx = \sum_{j=1}^q \sum_{k=0}^{n_j-1} A_{j}[k] \int_0^1 x^{k+s-1} w_j(x) dx + \sum_{j=1}^{p-q} \sum_{k=0}^{m_j-1} A_{q+j}[k] \int_0^1 x^{k+s-1} v_j(x) dx, $$
where
    $$ \int_0^1 x^{k+s-1} w_j(x) dx = \frac{\Gamma(s+k+\vec{a})}{\Gamma(s+k+\vec{b}+\vec{e}_j)} = 
    \frac{\Gamma(s+\vec{a})}{\Gamma(s+\vec{b})} \frac{(s+\vec{a})_k}{(s+\vec{b})_k (s+b_j+k)} $$
and
    $$ \int_0^1 x^{k+s-1} v_j(x) dx = \frac{\Gamma(s+k+\vec{a})}{\Gamma(s+k+\vec{b})} (k+s-1)^{j-1} = 
    \frac{\Gamma(s+\vec{a})}{\Gamma(s+\vec{b})} \frac{(s+\vec{a})_k}{(s+\vec{b})_k} (k+s-1)^{j-1}. $$
It then remains to show that 
$$\sum_{j=1}^q \sum_{k=0}^{n_j-1} A_{j}[k] \frac{(s+\vec{a})_k}{(s+\vec{b})_k (s+b_j+k)} + \sum_{j=1}^{p-q} \sum_{k=0}^{m_j-1} A_{q+j}[k] \frac{(s+\vec{a})_k}{(s+\vec{b})_k} (k+s-1)^{j-1} = \frac{p_{\sz{n}+\sz{m}-1}(s)}{\prod_{i=1}^q (s+b_i)_{n_i}}.$$
Multiplying the left-hand side by $\prod_{i=1}^q (s+b_i)_{n_i}$ yields
    $$\sum_{j=1}^q \sum_{k=0}^{n_j-1} A_{j}[k] (s+\vec{a})_k \prod_{i=1}^q \frac{(s+b_i)_{n_i}}{(s+b_i)_{k+\delta_{i,j}}} + \sum_{j=1}^{p-q} \sum_{k=0}^{m_j-1} A_{q+j}[k] (s+\vec{a})_k \prod_{i=1}^q \frac{(s+b_i)_{n_i}}{(s+b_i)_{k}} (k+s-1)^{j-1},$$
which is a polynomial if $n_i\geq n_j-1+\delta_{i,j}$ and $n_i\geq m_j-1$ for all $i,j$. The polynomial produced by the first terms has a degree of at most $\sz{n}+(p-q)\max\{n_j-1\mid j\}-1$, which is $\leq \sz{n}+\sz{m}-1 $ if $m_j\geq n_i-1$ for all $i,j$. The last two sets of conditions lead exactly to the near diagonal condition on $\vec{n}\sqcup\vec{m}$. The remaining terms produce a polynomial of degree at most $\sz{n}+\max\{(p-q)(m_j-1)+j \mid j\}-1$. The step-line condition on $\vec{m}$ then implies that this value is $\leq \sz{n}+\sz{m}-1$. Indeed, suppose that $m_j=m$ for $j\leq J$ and $m_j=m-1$ for $j>J$, then $\sz{m}=Jm+(p-q-J)(m-1)=(p-q)(m-1)+J$. If $j\leq J$, then $(p-q)(m-1)+j\leq \sz{m}$. If $J<j\leq p-q$, then $(p-q)(m-2)+j \leq (p-q)(m-1)+j-(p-q) \leq \sz{m}$.
\end{proof}

Similarly as in the proof of Proposition \ref{J_AT}, the composition formula \eqref{CompForm} and the fact that certain Jacobi-like systems of weights are AT-systems allow us to show that the same holds for certain Laguerre-like systems of weights.

\begin{prop} \label{L_AT}
    Suppose that all $b_j-a_j\in\Z_{\geq 0} \cup (\sz{n}+\sz{m}-1,\infty)$. Then the set of weights $\{w_j(x)\}_{j=1}^q\cup \{v_j(x)\}_{j=1}^{p-q}$ is an AT-system on $(0,\infty)$ for every $\vec{n}\sqcup\vec{m}\in\mathcal{N}^p$ with $\vec{m}\in\mathcal{S}^{p-q}$.
\end{prop}
\begin{proof}
Suppose that $F(x)$ is a polynomial combination of the described Laguerre-like weights that corresponds to the multi-index $\vec{n}\sqcup\vec{m}$. Its Mellin transform $\hat{F}(s)$ is then of the form 
    $$\frac{\Gamma(s+\vec{a})}{\Gamma(s+\vec{b}+\vec{n})} p_{\sz{n}+\sz{m}-1}(s),$$
see Lemma~\ref{L_PC}. We can decompose $\hat{F}(s)$ as $\hat{F}(s)/\Gamma(s+\vec{c}+\vec{m}) \cdot \Gamma(s+\vec{c}+\vec{m})$ for some $\vec{c}\in (-1,\infty)^{p-q}$ with $c_j-a_{q+j}>\sz{n}+\sz{m}-1$ for all $1\leq j\leq p-q$. The first factor corresponds to the Mellin transform of a polynomial combination of the Jacobi-like weights in $\{w_j(x;\vec{a},\vec{b}\sqcup\vec{c})\}_{j=1}^p$, see Lemma~\ref{J_PC}. We know from Proposition \ref{J_AT} that the latter is an AT-system for the multi-index $\vec{n}\sqcup\vec{m}$, hence the extended set of weights is a T-system. The second factor is the Mellin convolution of the gamma densities $\mathcal{G}^{c_j+m_j}(x)$, $1\leq j\leq p-q$. It is therefore sufficient to show that taking the consecutive convolution with the stated gamma densities preserves the property of being a T-system. Convoluting with a gamma density $\mathcal{G}^a(x)$ corresponds to taking $K(x,y)=(x/y)^a \exp(x/y)/y$ in the composition formula \eqref{CompForm}. The desired result then follows from the fact that the kernel $\exp(xy)$ is strictly totally positive of any order, see, e.g., \cite[Example 1 - Section 1.3]{KarlinStudden}.
\end{proof}

\begin{cor}
    Suppose that $b_j-a_j\in\Z_{\geq 0}$ for all $j$. Then the set of weights $\{w_j(x)\}_{j=1}^q\cup \{v_j(x)\}_{j=1}^{p-q}$ is an AT-system on $(0,\infty)$ for every multi-index $\vec{n}\sqcup\vec{m}\in\mathcal{N}^p$ with $\vec{m}\in\mathcal{S}^{p-q}$.
\end{cor}

Note that, compared to Corollary \ref{J_AT_sys} (the analogous result in the Jacobi-like setting), there is no condition of the form $b_i-b_j\not\in\Z$ whenever $i\neq j$. Such a condition was required in the former because the weights could reduce to the Jacobi-Piñeiro weights $(x^{b_1},\dots,x^{b_r})$ on $[0,1]$, but this is not the case here.
\medbreak

We can again derive a matrix Pearson equation for the system of weights. For the gamma density $\mathcal{G}^a(x)$, we have a Pearson equation $[\sigma(x) w(x)]' = \tau(x) w(x),$ with $\sigma(x)=x$ and $\tau(x)=(a+1)x$. Due to the limiting nature of the weights, the structure of the underlying matrix will be more involved than in the Jacobi-like setting.

\begin{prop}
Suppose that $b_i\neq b_j$ if $i\neq j$. Define $\vec{W} = (w_1,\dots,w_q,v_1,\dots,v_{p-q})^T$. Then
    $$[x\vec{W}(x)]' = \mathcal{T}(x) \vec{W}(x),\quad 
    \mathcal{T}(x) = \begin{pmatrix}
        \mathcal{T}_0 & \vec{1}\,^T & 0 \\
        0 & \vec{0}^T & I_{p-q-1} \\
        \vec{c} & x+d_0 & \vec{d}        
    \end{pmatrix},$$
where $\mathcal{T}_0 = -\text{diag}(b_j+1\mid 1\leq j\leq q)$, $c_k=-(\vec{a}-b_k)_1/(\vec{b}^{\ast k}-b_k)_1$ for $1\leq k\leq q$ and $d_{k}=- (d^k/ds^k)[(s+\vec{a})_1/(s+\vec{b})_1]_{s=1}/k!+(-1)^{k+1} \sum_{j=1}^q c_j/(b_j+1)^{k+1}$ for $0\leq k \leq p-q-1$.
\end{prop}
\begin{proof}
First, note that $(s+b_j) \hat{w}_j(s) = \hat{v}_1(s)$ and thus $(xw_j(x))' = - (b_j+1) w_j(x) + v_1(x)$ for all $1\leq j\leq q$. On the other hand, by definition, we have $(xv_{j}(x))'=v_{j+1}(x)$ for all $1\leq j\leq p-q-1$. It then remains to find the relation between $(xv_{p-q}(x))'$ and the other weights. Observe that
$$ \hat{v}_1(s+1) = \frac{\Gamma(s+\vec{a})}{\Gamma(s+\vec{b})} \frac{(s+\vec{a})_1}{(s+\vec{b})_1}, $$
in terms of a rational function that has a numerator of degree $p>q$ and a denominator of degree $q$. Hence, after a partial fraction decomposition, we find
$$\frac{(s+\vec{a})_1}{(s+\vec{b})_1} = - \sum_{j=1}^{q} \frac{c_j}{s+b_j} - \sum_{j=1}^{p-q} d_{j-1} (s-1)^{j-1} + (s-1)^{p-q}. $$
Consequently, 
    $$(s-1) \hat{v}_{p-q}(s) = \sum_{j=1}^{q} c_j \hat{w}_j(s) + \hat{v}_1(s+1) + \sum_{j=1}^{p-q} d_{j-1} \hat{v}_j(s), $$
and thus after taking the inverse Mellin transform
    $$ (xv_{p-q}(x))' =  \sum_{j=1}^{q} c_j w_j(x) + (x+d_0) v_1(x) + \sum_{j=2}^{p-q} d_{j-1} v_j(x). $$
The coefficients in the partial fraction decomposition can be computed as follows. Comparing residues in $s=-b_k$, we obtain $c_k=-(\vec{a}-b_k)_1/(\vec{b}^{\ast k}-b_k)_1$. On the other hand, by applying $(d^k/ds^k)[\,\cdot\,]_{s=1}/k!$, we get the stated expression for $d_k$.
\end{proof}

\subsection{Type I}
In this section, we will investigate the type I multiple orthogonal polynomials associated with the weights $(w_1,\dots,w_q,v_1,\dots,v_{p-q})$ introduced in the previous section for multi-indices $\vec{n}\sqcup\vec{m}$ near the diagonal with $\vec{m}$ on the step-line. Similarly as in the Jacobi-like setting, we will first derive an explicit expression for the Mellin transform of the corresponding type I functions. Essentially, this result also follows from Theorem \ref{J_MT} and the limiting relations between the weights; to remove the last component from $\vec{b}$ to get $\vec{b}^{\ast q}$, we can do the change of variables $x\mapsto x/b_q$ and let $b_q\to\infty$. 

\begin{thm} \label{L_MT}
    Let $\vec{n}\sqcup\vec{m}\in\mathcal{N}^p$ with $\vec{m}\in\mathcal{S}^{p-q}$. Every type~I function $F_{\vec{n}\sqcup\vec{m}}$ has a Mellin transform of the form
        $$\mathcal{F}_{\vec{n}\sqcup\vec{m}} \cdot \frac{\Gamma(s+\vec{a})}{\Gamma(s+\vec{b}+\vec{n})} (1-s)_{\sz{n}+\sz{m}-1},\quad \mathcal{F}_{\vec{n}\sqcup\vec{m}}\in\R.$$
\end{thm}
\begin{proof}
We know from Lemma \ref{L_PC} that the Mellin transform of $F_{\vec{n}\sqcup\vec{m}}$ is of the form 
    $$\frac{\Gamma(s+\vec{a})}{\Gamma(s+\vec{b}+\vec{n})} p_{\sz{n}+\sz{m}-1}(s),$$
in terms of a polynomial $p_{\sz{n}+\sz{m}-1}(s)$ of degree at most $\sz{n}+\sz{m}-1$. Due to the orthogonality conditions, we can conclude that this polynomial must be given by a scalar multiple of $(1-s)_{\sz{n}+\sz{m}-1}$.
\end{proof}

It is possible to be more generous with the restrictions on the multi-index if $a_j=b_j$ for some $j$, for
example, in the exponential integral case (where $(p,q)=(2,1)$, $\vec{a}=(\alpha,\alpha+\nu)$ and $\vec{b}=(\alpha+\nu$)) the above holds for multi-indices $(n,m)$ with $m+1\geq n$ (near and above the diagonal), see \cite[Lemma 2.1]{VAWolfs}. However, we won't go into further detail on this here.
\medbreak

In what follows, we will normalize the (non-zero) type I functions $F_{\vec{n}\sqcup\vec{m}}(x)$ such that $\mathcal{F}_{\vec{n}\sqcup\vec{m}} = 1$. 
\medbreak

Taking the inverse Mellin transform immediately yields an integral representation for the type I functions.

\begin{cor} \label{L_IR}
    Let $\vec{n}\sqcup\vec{m}\in\mathcal{N}^p$ with $\vec{m}\in\mathcal{S}^{p-q}$. Then,
    $$F_{\vec{n}\sqcup\vec{m}}(x) = \frac{1}{2\pi i} \int_{c-i\infty}^{c+i\infty} \frac{\Gamma(s+\vec{a})}{\Gamma(s+\vec{b}+\vec{n})} (1-s)_{\sz{n}+\sz{m}-1} x^{-s} ds,$$
    for some $c>0$ with $-\min\{a_j\mid j\} < c < 1$.
\end{cor}

We can derive the following Rodrigues-type formula from Theorem \ref{L_MT}.

\begin{cor}
Let $\vec{n}\sqcup\vec{m}\in\mathcal{N}^p$ with $\vec{m}\in\mathcal{S}^{p-q}$. Then,
    $$F_{\vec{n}\sqcup\vec{m}}(x) = \frac{d^{\sz{n}+\sz{m}-1}}{dx^{\sz{n}+\sz{m}-1}}\left[ x^{\sz{n}+\sz{m}-1} w_0(x;\vec{a},\vec{b}+\vec{n}) \right].$$
\end{cor}

The type I functions can be expressed as certain Mellin convolutions in which now gamma densities play a role.

\begin{prop}
Let $\vec{n}\sqcup\vec{m}\in\mathcal{N}^p$ with $\vec{m}\in\mathcal{S}^{p-q}$. Denote $\vec{a_\ast}=(a_{q+1}\dots,a_p)$. Then,
    $$F_{\vec{n}\sqcup\vec{m}} = F_{\vec{n}\sqcup\vec{m}}^{(\vec{a},\vec{b}\sqcup\vec{a}_\ast)} \ast \mathcal{G}^{m_1+a_{q+1}} \ast \dots \ast \mathcal{G}^{m_{p-q}+a_{p}}, $$
in terms of the type I function $F_{\vec{n}\sqcup\vec{m}}(x;\vec{a},\vec{b}\sqcup\vec{a}_\ast)$ from the Jacobi-like setting.
\end{prop}
\begin{proof}
The result follows from the multiplicative property of the Mellin convolution and the uniqueness of the Mellin transform; the Mellin transform of the type I functions in the Jacobi-like setting are given in Theorem \ref{J_MT}, the Mellin transform of a gamma density $\mathcal{G}^a(x)$ is $\Gamma(s+a)$.
\end{proof}

By working out the contour integral in Corollary \ref{L_IR} via the residue theorem, we can again find an expansion of the type I functions in terms of certain hypergeometric series (see Proposition \ref{J_HS}).

\begin{prop}
    Let $\vec{n}\sqcup\vec{m}\in\mathcal{N}^p$ with $\vec{m}\in\mathcal{S}^{p-q}$ and suppose that $a_i-a_j\not\in\Z$ whenever $i\neq j$ and $b_i-a_j\not\in\Z$ for all $i,j$. Set $N=\sz{n}+\sz{m}$. Then,
        $$F_{\vec{n}\sqcup\vec{m}}(x) = \sum_{j=1}^p \frac{(a_j+1)_{N-1} \Gamma(\vec{a}^{\ast j}-a_j)}{\Gamma(\vec{b}-a_j+\vec{n})} {}_{p+1}F_q \left( \begin{array}{c} N+a_j, -\vec{n}-\vec{b}+a_j+1 \\ a_j+1,-\vec{a}^{\ast j}+a_j+1 \end{array}; x \right) x^{a_j}. $$
\end{prop}

The terms in this sum (including the $x^{a_j}$) are all solutions to the differential equation for
    $${}_{q+1}F_p \left( \begin{array}{c} \sz{n}+\sz{m}, -\vec{n}-\vec{b}+1 \\ -\vec{a}+1 \end{array}; x \right),$$
see \cite[Eq. 16.8.6]{DLMF}. The type I function itself therefore satisfies the same differential equation.
\medbreak

Unfortunately, we can't use the same strategy as in the Jacobi-like setting to obtain formulas for the type I polynomials, because the series for the error in the underlying Hermite-Padé approximation isn't convergent anymore (though formally it would still make sense). In principle, one should be able to describe the polynomials inductively from the expressions found in the Jacobi-like setting by making use of the limiting relations in Lemma \ref{L_lim_rel}. Since doing so would lead to expressions that are rather involved, we won't do it here explicitly.

\subsection{Type II}

In what follows, we will study the type II multiple orthogonal polynomials associated with the weights $(w_1,\dots,w_q,v_1,\dots,v_{p-q})$ introduced in Section \ref{L_W} for multi-indices $\vec{n}\sqcup\vec{m}$ near the diagonal with $\vec{m}$ on the step-line. Formulas for these type II polynomials for multi-indices on the step-line were already obtained in \cite{Lima}, but we will derive them explicitly from the orthogonality conditions and for a larger set of multi-indices. They also follow immediately from Theorem \ref{J_II_HS} and the limiting relations; to go from $\vec{b}$ to a vector $\vec{b}^{\ast q}$ in which the last component is missing, we do the change of variables $x\mapsto x/b_q$ and let $b_q\to\infty$. Afterwards, we will describe the asymptotic zero distribution through free probability.

\begin{thm} \label{L_II_HS}
    Let $\vec{n}\sqcup\vec{m}\in\mathcal{N}^p$ with $\vec{m}\in\mathcal{S}^{p-q}$. Every type II polynomial $P_{\vec{n}\sqcup\vec{m}}(x)$ is of the form 
    $$\mathcal{P}_{\vec{n}\sqcup\vec{m}} \cdot {}_{q+1}F_p \left( \begin{array}{c} -\sz{n}-\sz{m}, \vec{b}+\vec{n}+1 \\ \vec{a}+1 \end{array}; x \right),\quad \mathcal{P}_{\vec{n}\sqcup\vec{m}}\in\R.$$
\end{thm}
\begin{proof}
Similarly as in the Jacobi-like setting (see Theorem \ref{J_II_HS}), it will be more natural to move into the complex plane and write
$$ P_{\vec{n}\sqcup\vec{m}}(x) =  \frac{1}{2\pi i} \int_{\Sigma} \frac{p_{\sz{n}+\sz{m}}(t)}{(-t)_{\sz{n}+\sz{m}+1}} \frac{\Gamma(t+\vec{b}+1)}{\Gamma(t+\vec{a}+1)} x^t dt,$$
in terms of a contour $\Sigma$ encircling the interval $[0,\sz{n}+\sz{m}]$ avoiding the poles of the integrand that are not in this interval and a polynomial $p_{\sz{n}+\sz{m}}(t)$ of degree at most $\sz{n}+\sz{m}$ that is determined by the relation
    $$\frac{p_{\sz{n}+\sz{m}}(t)}{(-t)_{\sz{n}+\sz{m}+1}} = \sum_{k=0}^{\sz{n}+\sz{m}} \frac{\Gamma(k+\vec{a}+1)}{\Gamma(k+\vec{b}+1)} \frac{P_{\vec{n}+\sz{m}}[k]}{t-k}. $$
The orthogonality conditions imply that 
 $$\int_{\Sigma} \frac{(t+\vec{a}+1)_s}{(t+\vec{b}+1)_s (s+t+b_j+1)} \frac{p_{\sz{n}+\sz{m}}(t)}{(-t)_{\sz{n}+\sz{m}+1}} dt = 0,\quad s=0,\dots,n_j-1,$$
$$ \int_{\Sigma} \frac{(t+\vec{a}+1)_s (s+t)^{j-1} }{(t+\vec{b}+1)_s} \frac{p_{\sz{n}+\sz{m}}(t)}{(-t)_{\sz{n}+\sz{m}+1}} dt = 0,\quad s=0,\dots,m_j-1.$$
As a consequence, since $\vec{n}\sqcup\vec{m}\in\mathcal{N}^p$ with $\vec{m}\in\mathcal{S}^{p-q}$, we must have 
    $$ \int_{\Sigma} \frac{p_{\sz{n}+\sz{m}}(t)}{\prod_{j=1}^q (t+b_j)_{n_j}} \frac{q_{\sz{n}+\sz{m}-1}(t)}{(-t)_{\sz{n}+\sz{m}+1}} x^t dt = 0,$$
for every polynomial $q_{\sz{n}+\sz{m}-1}(t)$ of degree at most $\sz{n}+\sz{m}-1$ (see also the proof of Lemma \ref{L_PC}). 
By considering the polynomials that have roots in all possible $\sz{n}+\sz{m}-1$ roots of $(-t)_{\sz{n}+\sz{m}+1}$ and applying the residue theorem, we can conclude that $p_{\sz{n}+\sz{m}}(t)/\prod_{j=1}^q (t+b_j)_{n_j}$ is constant along the $\sz{n}+\sz{m}+1$ roots of $(-t)_{\sz{n}+\sz{m}+1}$. We therefore must have that $p_{\sz{n}+\sz{m}}(t)$ is a scalar multiple of $\prod_{j=1}^q (t+b_j+1)_{n_j}$, from which the desired result then follows.
\end{proof}

From now on, we will assume that the (non-zero) type II polynomials $P_{\vec{n}\sqcup\vec{m}}(x)$ are normalized such that $\mathcal{P}_{\vec{n}\sqcup\vec{m}} = 1$.
\medbreak

The proof of Theorem \ref{L_II_HS} provides the following contour integral representation for the type II polynomials. Observe again the duality with Corollary \ref{L_IR}.

\begin{cor} \label{L_II_IR}
    Let $\vec{n}\in\mathcal{N}^p$ with $\vec{m}\in\mathcal{S}^{p-q}$. Then,
    $$P_{\vec{n}\sqcup\vec{m}}(x) = \frac{1}{2\pi i} \int_{\Sigma} \frac{1}{(-t)_{\sz{n}+\sz{m}+1}} \frac{\Gamma(t+\vec{b}+\vec{n}+1)}{\Gamma(t+\vec{a}+1)} x^t dt, $$
    where $\Sigma$ is a positively oriented contour that starts and ends at $\infty$ and encircles the positive real line, without enclosing any poles of the integrand not in $[0,\infty)$.
\end{cor}   

Some properties of the zeros of these type II polynomials for multi-indices on the step-line were already obtained in \cite{Lima} by exploiting the connection with branched continued fractions. We will give a proof by making use of Theorem \ref{J_RSZ} (the analogue in the Jacobi-like setting) and the finite free multiplicative convolution. Under the condition that all $b_j-a_j\in\Z_{\geq 0} \cup (\sz{n}+\sz{m}-1,\infty)$, these properties also follow from the fact that the underlying system of weights is an AT-system (see Proposition \ref{L_AT}). 

\begin{thm}[Lima \cite{Lima}, Theorem 4.8] \label{L_RSZ}
    Let $\vec{n}\sqcup\vec{m}\in\mathcal{S}^p$ and suppose that $b_j>a_i$ if $i\leq j$ and $b_j>a_i-1$ if $i>j$. The type~II polynomials $P_{\vec{n}\sqcup\vec{m}}(x)$ have real, positive, simple zeros.
\end{thm}
\textit{Alternative proof of Theorem \ref{L_RSZ}.}
Set $N=\sz{n}+\sz{m}$ and denote $\vec{a}_\ast = (a_{q+1},\dots,a_{p})$, then we have the following decomposition
\begin{align} \label{L_II_decomp}
    {}_{q+1}F_p &\left( \begin{array}{c} -N, \vec{b}+\vec{n}+1 \\ \vec{a}+1 \end{array}; x \right) = {}_{p+1}F_p\left( \begin{array}{c} -N, \vec{b}+\vec{n}+1, \vec{a}_\ast + \vec{m} + 1 \\ \vec{a}+1 \end{array}; x \right) \nonumber \\
    & \boxtimes_{N} {}_{1}F_1\left( \begin{array}{c} -N \\ a_{p-q+1}+m_1+1 \end{array}; x \right) \boxtimes_{N} \cdots \boxtimes_{N} {}_{1}F_1\left( \begin{array}{c} -N \\ a_{p}+m_{p-q}+1 \end{array}; x \right),
\end{align}
The first factor is $P_{\vec{n}\sqcup\vec{m}}(x;\vec{a},\vec{b}\sqcup\vec{a}_\ast)$, a type II polynomial in the Jacobi-like setting, and the others are Laguerre polynomials
$${}_{1}F_1\left( \begin{array}{c} -N \\ a_{q+j}+m_{j}+1 \end{array}; x \right) = L_{N}^{(a_{q+j}+m_{j})}(x).$$
Since each factor has real, positive, simple roots, the same holds for their finite free convolution (see \cite{SzegoConv} and \cite[Theorem 6]{K-S-V}).
\qed
\medbreak

We will now describe the asymptotic zero distribution of the (scaled) type II polynomials as a free convolution of measures. It will be more natural to use a different decomposition than \eqref{L_II_decomp} to achieve this. In principle, we could also use \eqref{L_II_decomp}, but then the expression for the asymptotic zero distribution would contain more factors than necessary and a later interpretation in terms of random matrices would be lost.

\begin{thm} \label{L_AZDConv}
    Suppose that $b_j>a_i$ if $i\leq j$ and $b_j>a_i-1$ if $i>j$. The asymptotic zero distribution of $P_{\vec{n}\sqcup\vec{m}}(N^{p-q}x)$ for $\vec{n}\sqcup\vec{m}\in\mathcal{S}^p$ as $N:=\sz{n}+\sz{m}\to\infty$ is given by $((1-1/p)\delta_1+(1/p)\mu_{p-1})^{\boxtimes q} \boxtimes \nu^{\boxtimes (p-q)} $ in terms of the deformed arcsin distribution $\mu_{p-1}$ from Lemma \ref{DefArcsin} and the Marchenko-Pastur distribution $\nu$.
\end{thm}
\begin{proof}
It is clear from \eqref{L_II_decomp} that the described scaling is indeed appropriate: the zeros of the free multiplicative convolution of polynomials are bounded by the product of the largest zeros of its factors (see \cite{SzegoConv}) and each factor $N$ in the scaling corresponds to an appearance of a Laguerre polynomial. Additionally, it follows that the asymptotic zero distribution is independent of the precise value of the parameters. We can therefore assume that $\vec{b}=\vec{a}^\ast$ with $\vec{a}^\ast=(a_1,\dots,a_p)$. The decomposition
$$\begin{aligned}
    &{}_{q+1}F_p\left( \begin{array}{c} -N, \vec{a}+\vec{n}+1 \\ \vec{a}+1 \end{array}; N^{p-q} x \right) \\
    &=  {}_{2}F_1\left( \begin{array}{c} -N, a_1+n_1+1 \\ a_1+1 \end{array}; x \right) \boxtimes_{N} \cdots \boxtimes_{N} {}_{2}F_1\left( \begin{array}{c} -N, a_q+n_q+1 \\ a_q+1 \end{array}; x \right)
    \\
    & \quad \boxtimes_{N} {}_{1}F_1\left( \begin{array}{c} -N \\ a_{q+1}+1 \end{array}; Nx \right) \boxtimes_{N} \cdots \boxtimes_{N} {}_{1}F_1\left( \begin{array}{c} -N \\ a_p+1 \end{array}; Nx \right),
\end{aligned}$$
in terms of the Jacobi polynomials $P_N^{(a_j,n_j-N)}(x)=(1-x)^{N-n_j} P_{n_j}^{(a_j,N-n_j)}(x)$ and scaled Laguerre polynomials $L_N^{(a_{q+j})}(Nx)$, together with Lemma \ref{DefArcsin} and the fact that the Marchenko-Pastur distribution $\nu$ is the asymptotic zero distribution of the scaled Laguerre polynomials, then leads to the desired result.
\end{proof}

By making use of the previous result, we can derive an algebraic equation for the Stieltjes transform of the asymptotic zero distribution. Alternatively, we can start from the hypergeometric differential equation for the type II polynomials, see \cite[Eq. (16.8.3)]{DLMF} and employ the usual techniques (see, e.g., \cite[Section 3.1.4]{LeursVA} or \cite[Section 3.4.2]{VAWolfs}). Generally, the described algebraic equation doesn't seem to be solvable explicitly.

\begin{cor}
    Suppose that $b_j>a_i$ if $i\leq j$ and $b_j>a_i-1$ if $i>j$. The Stieltjes transform of the asymptotic zero distribution of $P_{\vec{n}\sqcup\vec{m}}(N^{p-q}x)$ as $N:=\sz{n}+\sz{m}\to\infty$ satisfies the algebraic equation $(zS(z))^{p+1} = z (zS(z)+1/p)^q (zS(z)-1)$.
\end{cor}
\begin{proof}
We can prove this by working out the $S$-transform of the asymptotic zero distribution in terms of the $S$-transforms of its factors $(1-1/p)\delta_1+(1/p)\mu_{p-1}$ and $\nu$. Afterwards, we can use the fact that a Stieltjes transform $S(z)=\sum_{k\geq 0} m_k z^{-k-1} = (1+m(1/z))/z$ and $S$-transform $s(z)=m^{-1}(z) (z+1)/z$ are connected through the moment generating function $m(z)=\sum_{k\geq 0} m_k z^k$. The $S$-transform of $(1-1/p)\delta_1+(1/p)\mu_{p-1}$ and $\nu$ is given by $s_1(z)=1+1/(p(z+1))$ and $s_2(z)=1/(z+1)$ respectively. This follows from the algebraic equations for the corresponding Stieltjes transforms, which can be obtained from the hypergeometric differential equation for $P_n^{(a,(p-1)n)}(x)$ and $L_n^{(a)}(nx)$ in the standard way. 
\end{proof}

\section{Bessel-like setting}

\subsection{System of weights} \label{B_W}

Let $p<q$ and suppose that $\vec{a}\in(-1,\infty)^p$ and $\vec{b}\in(-1,\infty)^q$ satisfy $a_j < b_j$ for $1\leq j \leq p$. We will consider the $q$ weights $w_j(z;\vec{a},\vec{b})=f_0(z;\vec{a},\vec{b}+\vec{e}_j)$, with $f_0$ as in \eqref{Intro_HS}, on the unit circle $\{z\in\C:|z|\,=1\}$. In this setting, such weights are convergent for every non-zero $z\in\C$. It follows from properties of the Mellin transform that there is no weight $w_j(z;\vec{a},\vec{b})$ on the positive real line with a Mellin transform that leads to moments of the form
$$ \int_{|z|=1} z^{s-1} w_j(z;\vec{a},\vec{b}) dz = \frac{\Gamma(s+\vec{a})}{\Gamma(s+\vec{b}+\vec{e}_j)}.$$

The limiting relations between these kinds of weights are straightforward: we may remove the last component in the vector $\vec{a}$ via
    $$\lim_{a_q\to\infty} \frac{1}{\Gamma(a_q)} w_j(a_q z;\vec{a},\vec{b}) = w_j(z;\vec{a}^{\ast p},\vec{b}). $$
We can also use it to convert a set of Jacobi-like weights from Section \ref{J_W} to a set of Bessel-like weights. This is because the Jacobi-like weights can also be interpreted as weights on the unit circle. Indeed, if $\vec{a},\vec{b}\in (-1,\infty)^r$, we can define $w_j(z;\vec{a},\vec{b}) = f_0(1/z;\vec{a},\vec{b}+\vec{e}_j)/z$ (it is convergent in $\{z\in\C: 0<\left|z\right|\leq 1\}$) and still have the desired moments.
\medbreak

Since we are working with weights on the unit circle, the notion of an AT-system isn't well-defined anymore. The matrix Pearson equation for this set of weights is similar to the one in the Jacobi-like setting, see Proposition \ref{J_W_MPE}.
\begin{prop} \label{B_W_MPE}
Define $\vec{W} = (w_1,\dots,w_q)^T$. Then
    $$[z(1-z)\vec{W}(x)]' = \mathcal{T}(z) \vec{W}(z),$$
where
     $$\mathcal{T}_{k,j}(z) = - \frac{\prod_{i=1}^p (a_i-b_j)}{\prod_{i\neq j}^q (b_i-b_j)} + \begin{dcases}
            0,\ k\neq j, \\
            z - (b_j+1)(1-z),\ k=j.
        \end{dcases}$$
\end{prop}
\begin{proof}
    We can use the same strategy as in the Jacobi-like setting (Proposition \ref{J_W_MPE}). The results obtained for the Mellin transforms can be interpreted as results about the coefficients in the expansions of the weights here.
\end{proof}

\subsection{Type I polynomials}

Below, we will study the type I multiple orthogonal polynomials associated with the weights $(w_1,\dots,w_q)$ from the previous section for multi-indices $\vec{n}$ near the diagonal. Since, in this setting, the notion of Mellin transform is not well-defined anymore, we will only be able to compute the moments of the type I functions. However, this is sufficient to obtain formulas for the associated type I polynomials. Other results, like the differential formula, convolution formula and hypergeometric expansion don't necessarily follow anymore.
\medbreak

We will first compute the moments of polynomial combinations of the weights.

\begin{lem}
Let $\vec{n}\in\mathcal{N}^q$. The $(s-1)$-th moment of $F(z) = \sum_{j=1}^q A_{j}(z) w_j(z)$ with $\deg A_{j} \leq n_j-1$ is of the form
    $$\frac{\Gamma(s+\vec{a})}{\Gamma(s+\vec{b}+\vec{n})} p_{\sz{n}-1}(s),$$ 
where $p_{\sz{n}-1}(s)$ is a polynomial of degree at most $\sz{n}-1$.
\end{lem}
\begin{proof}
We can follow the proof of Lemma \ref{J_PC} (the analogue in the Jacobi-like setting). The degree of the polynomial in \eqref{J_PC_poly} is at most $\sz{n}-1$, because the degree of the term that corresponds to the indices $(j,k)$ is $k(p-q)+\sz{n}-1$ and this is at most $\sz{n}-1$ for $0\leq k\leq n_j-1$ since $p<q$.
\end{proof}

We can obtain an explicit expression for the moments of the type I functions in the same way as in the Jacobi-like setting, see Theorem \ref{J_MT}. It is also implied by the limiting relations between the weights: we may remove the last component from $\vec{a}$ to get $\vec{a}^{\ast p}$ by doing the change of variables $a_p\mapsto a_p z$ and letting $a_p\to\infty$.

\begin{thm} \label{B_MT}
Let $\vec{n}\in\mathcal{N}^q$. The $(s-1)$-th moment of every type I function $F_{\vec{n}}$ is of the form
    $$\mathcal{F}_{\vec{n}} \cdot \frac{\Gamma(s+\vec{a})}{\Gamma(s+\vec{b}+\vec{n})} (1-s)_{\sz{n}-1},\quad \mathcal{F}_{\vec{n}}\in\R.$$ 
\end{thm}

In what follows, we will assume that the (non-zero) type I functions $F_{\vec{n}}(x)$ are normalized such that $\mathcal{F}_{\vec{n}} = 1$.
\medbreak

Theorem \ref{B_MT} allows us to obtain explicit expressions for the type I polynomials.

\begin{thm} \label{B_I_poly}
    Let $\vec{n}\in\mathcal{N}^q$ and suppose that $b_j-a_i\not\in\Z_{<0}$ and $b_i-b_j\not\in\Z_{\geq 0}$ whenever $i\neq j$. The type I polynomials $A_{\vec{n},j}$ are given by
    $$ A_{\vec{n},j}(z) = \frac{(\vec{a}-b_j)_1}{(\vec{b}-b_j)_1^\ast} \sum_{J=1}^q \sum_{K=0}^{n_J-1} P_{\vec{n},J}[K] \sum_{k=0}^{K-1+\delta_{j,J}} \frac{(\vec{b}-b_J-K)_{k+1}^{\ast j} (b_j-b_J-K)_k}{(\vec{a}-b_J-K)_{k+1}} z^k, $$
    where
    $$ P_{\vec{n},J}[K] = \frac{(b_J+K+1)_{\sz{n}-1}}{\prod_{i=1,i\neq J}^q (b_i-b_J-K)_{n_i}} \frac{(-1)^K}{K!(n_J-K-1)!} .$$
\end{thm}
\begin{proof}
The proof is analogous to that in the Jacobi-like setting, see Theorem \ref{J_I_poly}. The essential partial fraction decomposition \eqref{J_I_poly_PFD} still goes through.
\end{proof}

\begin{cor}
    Suppose that $b_j-a_i\not\in\Z_{<0}$ and $b_i-b_j\not\in\Z_{\geq 0}$ whenever $i\neq j$. Then every multi-index near the diagonal is normal for the set of weights $\{w_j(z)\}_{j=1}^q$.
\end{cor}

\subsection{Type II polynomials}
In this section, we will determine formulas for the type II multiple orthogonal polynomials associated with the weights $(w_1,\dots,w_q)$ introduced in Section \ref{B_W} for multi-indices near the diagonal. These formulas were already obtained in \cite{Lima}, but we will provide a constructive proof that holds for a larger set of multi-indices. They immediately follow from the limiting relations as well; to go from $\vec{a}$ to $\vec{a}^{\ast p}$ and remove the last component, we do the change of variables $z\mapsto a_p z$ and let $a_p\to\infty$.

\begin{thm}
    Let $\vec{n}\in\mathcal{N}^q$. Every type II polynomial $P_{\vec{n}}(z)$ is of the form
    $$\mathcal{P}_{\vec{n}} \cdot {}_{q+1}F_p \left( \begin{array}{c} -\sz{n}, \vec{b}+\vec{n}+1 \\ \vec{a}+1 \end{array}; z \right),\quad \mathcal{P}_{\vec{n}}\in\R.$$
\end{thm}
\begin{proof}
We can use the same idea as in the Jacobi-like setting, see Theorem \ref{J_II_HS}. Since $p<q$, the integrand of the contour integral \eqref{J_II_HS_CI} is still $\BO(t^{-2})$ as $t\to\infty$.
\end{proof}

The zeros of these polynomials have not been studied before in full generality. In the case of the Bessel polynomials, some properties of the zeros are known, see \cite[Section 4.10]{Ismail} for an overview, but such results are not straightforward to obtain because the zeros accumulate on a curve in the complex plane. The approach with the finite free multiplicative convolution doesn't seem to lead to new results either, because decompositions always involve polynomials that can have roots in the complex plane (even in the case $\vec{b}=\vec{a}\sqcup\vec{c}$). For example, we can write
    $$\begin{aligned}
    &{}_{q+1}F_p\left( \begin{array}{c} -\sz{n}, \vec{b}+\vec{n}+1 \\ \vec{a}+1 \end{array}; z \right) \\
    &= {}_{2}F_1\left( \begin{array}{c} -\sz{n}, b_1+n_1+1 \\ a_1+1 \end{array}; z \right) \boxtimes_{\sz{n}} \cdots \boxtimes_{\sz{n}} {}_{2}F_1\left( \begin{array}{c} -\sz{n}, b_p+n_p+1 \\ a_p+1 \end{array}; z \right)
    \\
    & \quad \boxtimes_{\sz{n}} {}_{2}F_0\left( \begin{array}{c} -\sz{n}, b_{p+1}+n_{p+1}+1 \\ - \end{array}; z \right) \boxtimes_{\sz{n}} \cdots \boxtimes_{\sz{n}} {}_{2}F_0\left( \begin{array}{c} -\sz{n}, b_{q} + n_q + 1 \\ - \end{array}; z \right),
\end{aligned}$$
which involves the Jacobi polynomials $P_{\sz{n}}^{(a_j,b_j-a_j+n_j-\sz{n})}(z)$ and the Bessel polynomials $B_{\sz{n}}^{(b_j+n_{j}-\sz{n})}(z)$, both in a typically negative non-integer parameter.

\section{Applications} 
\subsection{Products of random matrices}

We will establish the connection between the multiple orthogonal polynomials in the Jacobi/Laguerre-like setting and products of truncated Haar distributed unitary random matrices and Ginibre matrices (matrices with independent standard complex Gaussian random variables as its entries). In the most basic cases, where we would only consider a Ginibre matrix or a truncated unitary random matrix, it is well-known that there is a connection with the Laguerre and Jacobi orthogonal polynomials, which are the most basic cases of our multiple orthogonal polynomials (see, e.g., \cite{Kuijlaars3}). The connection arises in the following way. Let $\Delta_n(\vec{x})=\prod_{1\leq i<j\leq n} (x_j-x_i)$ denote the Vandermonde determinant associated with the vector $\vec{x}\in\R^n$. It is known that the joint probability density of the squared singular values of these random matrices is given by a polynomial ensemble
    $$ \propto \Delta_n(\vec{x}) \det\left[w_k(x_j)\right]_{j,k=1}^n,\quad \vec{x}\in\R^n $$
(the symbol $\propto$ means proportional to). For a size $(n+a) \times n$ upper left corner of a truncated unitary random matrix of size $(n+b)\times (n+b)$ (assuming that $b\geq n+a+1$), the weight functions are supported on $(0,1)$ and given by beta densities $w_k(x) = \mathcal{B}^{a+k-1,b-n}(x)$. 
For a size $(n+a) \times n$ Ginibre matrix, the weight functions are supported on $(0,\infty)$ and given by gamma densities $w_k(x) = \mathcal{G}^{a+k-1}(x)$. Note that the latter arises as an appropriate limit of the former by doing a suitable scaling and letting $b\to\infty$.

Typically one is interested in the correlation kernel of such polynomial ensembles
    $$ K_n(x,y) = \sum_{k=0}^{n-1} P_k(x)Q_k(y).$$
The kernel may be built out of monic polynomials $P_k$ of degree $k$ and functions $Q_l$ in the linear span of $w_1,\dots,w_{l+1}$ that satisfy the biorthogonality conditions
\begin{equation} \label{biorth_PQ}
    \int_0^\infty P_k(x) Q_l(x) dx = \delta_{k,l},\quad k,l=0,\dots,n-1.
\end{equation}
These are satisfied whenever
\begin{equation} \label{orth_P}
    \int_0^\infty P_k(x) w_j(x) dx = 0,\quad j=0,\dots,k-1,
\end{equation}
\begin{equation} \label{orth_Q}
    \int_0^\infty Q_l(x) x^j dx = \delta_{l,j},\quad j=0,\dots,l.
\end{equation}    
Such a construction may not always be possible, however if it is, the $P_k$ and $Q_l$ are uniquely determined. Constructions using other types of biorthogonal functions $P_k$ and $Q_l$ are also allowed as long as they span the spaces $\text{span}\{1,\dots,x^{n-1}\}$ and $\text{span}\{w_1,\dots,w_{n}\}$ respectively. Another object of interest is the monic polynomial $P_n(x)$ of degree $n$ that extends the biorthogonality conditions \eqref{biorth_PQ} to $k=n$, because it is the average characteristic polynomial $ \mathbb{E}[ \prod_{j=1}^n (x-x_j)] $ over the particles in the associated polynomial ensemble. 

For the described truncated unitary random matrix, it can be shown that
    $$P_k(x) = P_k^{(a,b-n)}(x),\quad Q_{k-1}(x) = \mathcal{B}^{a,b-n}(x) P_{k}^{(a,b-n)}(x),\quad k=0,\dots,n,$$
in terms of Jacobi polynomials. On the other hand, for the stated Ginibre matrix, we have
    $$P_k(x) = L_k^{(a)}(x),\quad Q_{k-1}(x) = \mathcal{G}^{a}(x) L_{k}^{(a)}(x),\quad k=0,\dots,n,$$
in terms of Laguerre polynomials. We will show that in order to study the kernel associated with products of these random matrices, we essentially have to take finite free convolutions of the underlying polynomials $P_k$ and Mellin convolutions of the underlying functions $Q_l$. 
\medbreak

The joint probability density of the squared singular values of products of truncated unitary random matrices and Ginibre matrices can be described using the following results.
\begin{thm}[Kieburg-Kuijlaars-Stivigny \cite{KieburgKuijlaarsStivigny}, Corollary 2.4]
Let $T$ be a size $(n+\nu)\times l$ truncation of a Haar distributed unitary matrix of size $m\times m$ with $\nu\geq 0$, $m\geq l\geq n\geq 1$ and $m\geq n+\nu+1$. Let $X$ be a random matrix of size $l\times n$, independent of $T$, such that the joint probability density of its squared singular values is given by a polynomial ensemble on $(0,\infty)^n$ in terms of functions $f_1,\dots,f_n$. Then the joint probability density of the squared singular values of $Y=TX$ is given by a polynomial ensemble on $(0,\infty)^n$ in terms of functions $g_1,\dots,g_n$ with
    $$g_k(y) = \int_0^1 x^\nu (1-x)^{m-n-\nu-1} f_k\left(\frac{y}{x}\right) \frac{dx}{x},\quad k=1,\dots,n.  $$
\end{thm}

\begin{thm}[Kuijlaars-Stivigny \cite{KuijlaarsStivigny}, Theorem 2.1]
Let $G$ be a Ginibre matrix of size $(n+\nu)\times l$ with $\nu\geq 0$ and $l\geq n\geq 1$. Let $X$ be a random matrix of size $l\times n$, independent of $G$, such that the joint probability density of its squared singular values is given by a polynomial ensemble on $(0,\infty)^n$ in terms of functions $f_1,\dots,f_n$. Then the joint probability density of the squared singular values of $Y=GX$ is given by a polynomial ensemble on $(0,\infty)^n$ in terms of functions $g_1,\dots,g_n$ with
    $$g_k(y) = \int_0^\infty x^\nu e^{-x} f_k\left(\frac{y}{x}\right) \frac{dx}{x},\quad k=1,\dots,n.  $$
\end{thm}

We will first establish the connection between certain products of truncated unitary matrices and the multiple orthogonal polynomials in the Jacobi-like setting (studied in Section~2, see Theorem \ref{J_MT} and Theorem \ref{J_II_HS}). Consider a multi-index $\vec{n}\in\N^r$ with $n=\sz{n}$. Suppose that $\vec{a},\vec{b}\in\N^r$ with $b_j\geq (n-n_1)\delta_{j,1}+a_j$ for all $j$ and set $a_0=0$. Let $T_j$ be a size $(n+a_j)\times (n+a_{j-1})$ truncation of a size $(n+n_j+b_j)\times (n+n_j+b_j)$ Haar distributed unitary matrix. Then the joint probability density of the squared singular values of $T_r\cdots T_1$ is given by
\begin{equation} \label{RM_T}
    \propto \Delta_n(x) \det\left[ w_0(x_j;\vec{a}+(k-1)\vec{e}_1,\vec{b}+\vec{n}+(k-n-1)\vec{e}_1) \right]_{j,k=1}^n.
\end{equation}
with $w_0$ as in Section \ref{J_W}.

\begin{prop} \label{RM_T_MOP}
    The biorthogonal functions in \eqref{orth_P} and \eqref{orth_Q} associated with the squared singular values of the product $T_r\cdots T_1$ described above are given by 
    $$P_k(x) = P_{\vec{k}}(x;\vec{a},\vec{b}+\vec{n}-\vec{k}),\quad Q_{k-1}(x) = F_{\vec{k}}(x;\vec{a},\vec{b}+\vec{n}-\vec{k}),\quad k=0,\dots,n,$$
    in terms of type II multiple orthogonal polynomials and type I functions in the Jacobi-like setting, for any $\vec{k}\in\mathcal{N}^r$ with $|\vec{k}|=k$.
\end{prop}
\begin{proof}
   It was proven in \cite[Proposition 2.7]{KieburgKuijlaarsStivigny} that 
        $$P_k(x) = \frac{1}{2\pi i} \int_{\Sigma_k} \frac{1}{(t-k)_{k+1}} \frac{\Gamma(t+\vec{b}+\vec{n}+1)}{\Gamma(t+\vec{a}+1)} x^t dt,\quad k=0,\dots,n-1,$$
    but it also extends to $k=n$, for a positively oriented contour $\Sigma_k$ encircling $[0,k]$ once and not enclosing any pole of the integrand not in this interval, and
        $$Q_l(x) = \frac{1}{2\pi i} \int_{C} (s-l)_l \frac{\Gamma(s+\vec{a})}{\Gamma(s+\vec{b}+\vec{n})} x^{-s} ds,\quad k=0,\dots,n-1, $$
    for a positively oriented contour $C$ that starts and ends at $-\infty$ and encircles the negative real line. Hence, the result follows after comparison with the integral representations for the type II multiple orthogonal polynomials and type I functions in the Jacobi-like setting, see Corollaries \ref{J_II_IR} and \ref{J_IR}.
\end{proof}

\begin{cor} \label{RM_T_E}
    If $\vec{n}\in\mathcal{N}^r$, then $P_{\vec{n}}(x;\vec{a},\vec{b})=\mathbb{E}[ \prod_{j=1}^n (x-x_j)]$ with the expectation taken over the density \eqref{RM_T}.
\end{cor}

It is remarkable that a condition of the form $b_1\geq (n-n_1)+a_j$ from Proposition \ref{J_AT} is required here in this construction as well. Since already $b_j-a_j\in\Z_{\geq 0}$ in this setting, the system of weights associated with $P_{\vec{n}}(x;\vec{a},\vec{b})$ is therefore necessarily an AT-system.
\medbreak

The multiple orthogonal polynomials in the Laguerre-like setting will become relevant if we also include Ginibre matrices in the product: the results here should be compatible with the limiting relations between the multiple orthogonal polynomials and the random matrices that were mentioned before. Let $p>q$ and consider multi-indices $\vec{n}\in\N^{q},\vec{m}\in\N^{p-q}$ with $n=\sz{n}+\sz{m}$. Suppose that $\vec{a},\vec{b}\in\N^q,\vec{c}\in\N^{p-q}$ with $a_j\leq b_j$ for all $j$. Set $a_0=c_{p-q}$ and $c_0=0$. Let $T_j$, $1\leq j\leq q$, be a size $(n+a_j)\times (n+a_{j-1})$ truncation of a size $(n+n_j+b_j)\times (n+n_j+b_j)$ Haar distributed unitary matrix. Let $G_j$, $1\leq j\leq p-q$, be a size $(n+c_j)\times (n+c_{j-1})$ Ginibre matrix. Then the joint probability density of the squared singular values of $T_q\cdots T_1 \cdot G_{p-q} \cdots G_1 $ is given by
\begin{equation} \label{RM_TG}
    \propto \Delta_n(x) \det\left[ w_0(x_j;\vec{a}\sqcup\vec{c}+(k-1)\vec{e}_{q+1},\vec{b}+\vec{n}) \right]_{j,k=1}^n.
\end{equation}
with $w_0$ as in Section \ref{L_W}.

\begin{prop} \label{RM_TG_MOP}
    The biorthogonal functions in \eqref{orth_P} and \eqref{orth_Q} associated with the squared singular values of the product $T_q\cdots T_1 \cdot G_{p-q} \cdots G_1 $ described above are given by 
        $$P_k(x) = P_{\vec{k}_1\sqcup\vec{k}_2}(x;\vec{a}\sqcup\vec{c},\vec{b}+\vec{n}-\vec{k}_1),\quad Q_{k-1}(x) = F_{\vec{k}_1\sqcup\vec{k}_2}(x;\vec{a}\sqcup\vec{c},\vec{b}+\vec{n}-\vec{k}_1),\quad k=0,\dots,n,$$
    in terms of type II multiple orthogonal polynomials and type I functions in the Laguerre-like setting, for any $\vec{k}_1\sqcup\vec{k}_2\in\mathcal{N}^{p}$, for which $\vec{k}_2\in\mathcal{S}^{p-q}$, with $|\vec{k}_1|\,+\,|\vec{k}_2|\,=k$.
\end{prop}
\begin{proof}

The Mellin transform of $Q_l$ is of the form
    $$ \hat{Q}_l(s) = \sum_{j=1}^{l+1} c_l \hat{w}_0(s;\vec{a}\sqcup\vec{c}+(j-1)\vec{e}_{q+1},\vec{b}+\vec{n}) = \frac{\Gamma(s+\vec{a}\sqcup\vec{c})}{\Gamma(s+\vec{b}+\vec{n})} q_l(s) $$
in terms of a polynomial $q_l(x)$ of degree at most $l$. The conditions in \eqref{orth_Q} then imply that $q_l(x)\propto (s-l+1)_{l-1} $ and that $\hat{Q}_l(l)=1$. Comparing with the Mellin transform of the type I functions in the Laguerre-like setting (see Theorem \ref{L_MT}) then leads to the desired result for the functions $Q_l$. Similarly as in the proofs of \cite[Proposition 2.7]{KieburgKuijlaarsStivigny} (for a product $T_r\dots T_1$) and \cite[Proposition 4.2 and 4.3]{KuijlaarsStivigny} (for a product $G_{M-1}\dots G_1 T_1$), we can then verify that
    $$P_k(x) = \frac{1}{2\pi i} \int_{\Sigma_k} \frac{1}{(t-k)_{k+1}} \frac{\Gamma(t+\vec{b}+\vec{n}+1)}{\Gamma(t+\vec{a}\sqcup\vec{c}+1)} x^t dt,\quad k=0,\dots,n.$$
Hence, after comparison with the integral representation for the type II multiple orthogonal polynomials in the Laguerre-like setting (see Corollary \ref{L_II_IR}), we obtain the result about the polynomials $P_k$ as well.
\end{proof}

\begin{cor} \label{RM_TG_E}
    If $\vec{n}\sqcup\vec{m}\in\mathcal{N}^p$ and $\vec{m}\in\mathcal{S}^{p-q}$, then $P_{\vec{n}\sqcup\vec{m}}(x;\vec{a}\sqcup\vec{c},\vec{b})=\mathbb{E}[ \prod_{j=1}^n (x-x_j)]$ with the expectation taken over the density \eqref{RM_TG}.
\end{cor}

Corollaries \ref{RM_T_E} and \ref{RM_TG_E} imply that the zeros of the type II multiple orthogonal polynomials correspond to the zeros of the average characteristic polynomial of the squared singular values of a product of truncated unitary and Ginibre matrices. This connection leads to certain properties of the zeros of the type II polynomials (like simplicity and their location), at least under appropriate conditions on the parameters. Moreover, one can interpret the asymptotic zero distribution of the type II multiple orthogonal polynomials as the asymptotic distribution of the squared singular values of the corresponding product of truncated unitary and Ginibre matrices. This was to be expected because the latter arises as the free convolution of the asymptotic distribution of the squared singular values of its truncated unitary and Ginibre matrix factors, which corresponds to the deformed arcsin distributions and Marchenko-Pastur distributions that appear in the former (see Theorem \ref{J_AZDConv} and Theorem \ref{L_AZDConv}). A result of a similar kind, but for (non-random) normal matrices $A,B\in\C^{d\times d}$, was proven in \cite[Theorem 1.5]{M-S-S}. One showed that
$$ \det(xI_d-A) \boxtimes_d \det(xI_d-B) = \mathbb{E}[\det(xI_d-AUBU^\ast)], $$
where the expectation is taken over a $d\times d$ Haar distributed unitary random matrix $U$. 

\medbreak

The integral representations in the proofs of Propositions \ref{RM_T_MOP} and \ref{RM_TG_MOP} allow us to derive a double integral representation for the associated correlation kernel.

\begin{prop}
    In the setting of Propositions \ref{RM_T_MOP} and \ref{RM_TG_MOP}, we have
        $$K_n(x,y) = \frac{1}{(2\pi i)^2} \int_C \int_\Sigma \frac{(s-n)_n}{(t-n+1)_n} \frac{\Gamma(s+\vec{a}\sqcup\vec{c})}{\Gamma(t+\vec{a}\sqcup\vec{c}+1)} \frac{\Gamma(t+\vec{b}+\vec{n}+1)}{\Gamma(s+\vec{b}+\vec{n})} \frac{x^{t-s}}{s-t-1} dt ds,$$
    where $\Sigma$ is a positively oriented contour that starts and ends at $\infty$ and encircles the positive real line, without enclosing any poles of the integrand not in $[0,\infty)$, and $C$ is a positively oriented contour that starts and ends at $-\infty$ and encircles the negative real line.
\end{prop}
\begin{proof}
The integral representations in the proofs of Propositions \ref{RM_T_MOP} and \ref{RM_TG_MOP} show that
    $$K_n(x,y) = \frac{1}{(2\pi i)^2} \int_C \int_\Sigma \sum_{k=0}^{n-1} \frac{(s-k)_k}{(t-k)_{k+1}} \frac{\Gamma(s+\vec{a}\sqcup\vec{c})}{\Gamma(t+\vec{a}\sqcup\vec{c}+1)} \frac{\Gamma(t+\vec{b}+\vec{n}+1)}{\Gamma(s+\vec{b}+\vec{n})} x^{t-s} dt ds. $$
    Then we can use \cite[Eq. (6.5)]{KuijlaarsStivigny}, which is 
        $$\sum_{k=0}^{n-1} \frac{(s-k)_k}{(t-k)_{k+1}} = \frac{1}{s-t-1} \left( \frac{(s-n)_n}{(t-n+1)_n} - 1 \right).$$
    The double integral that corresponds to the second term vanishes, because the contour $\Sigma$ doesn't contain any poles of the integrand in the variable $t$. Hence we immediately obtain the desired result. 
\end{proof}

In particular, this result allows us to obtain hard edge scaling limits at the origin as $n=\sz{n}+\sz{m}\to\infty$. It would lead to the different kernels described in \cite[Theorem 2.8]{KuijlaarsStivigny} that arose as scaling limits at the origin of kernels associated with certain products of truncated unitary matrices.

\subsection{Diophantine approximation}

In this section, we will use the connection between multiple orthogonal polynomials and Hermite-Padé approximation (see, e.g., \cite[Chapter 23]{Ismail}) to study the Diophantine properties of values of certain hypergeometric series. Generally, Hermite-Padé approximation is used to simultaneously approximate several functions $f_j(z)$ with an expansion of the form 
    $$ f_j(z) = \sum_{k\geq 0} \frac{m_{j,k}}{z^{k+1}},\ z\to\infty. $$
Consider $p\leq q$ and suppose that $\vec{a}\in(-1,\infty)^p$ and $\vec{b}\in(-1,\infty)^q$ satisfy $a_j<b_j$ for all $1\leq j \leq p$. The multiple orthogonal polynomials in the Jacobi- and Bessel-like setting will appear in the Hermite-Padé approximation problems associated with the $q$ functions $f_j(z;\vec{a},\vec{b})=f_0(1/z;\vec{a},\vec{b}+\vec{e}_j)/z$, with $f_0$ as in \eqref{Intro_HS}. The series $f_j(z;\vec{a},\vec{b})$ are convergent in $\C_0$, if $p<q$, and in $\{z\in\C: 0<\left|z\right|\leq 1\}$, if $p=q$. The associated type I Hermite-Padé approximation problem is given by
    $$ \sum_{j=1}^q A_{\vec{n},j}(z) f_j(z;\vec{a},\vec{b}) - B_{\vec{n}}(z) = \BO(z^{-\sz{n}}),\ z\to\infty, $$
in terms of polynomials $A_{\vec{n},j}(z)$ and $B_{\vec{n}}(z)$ of degree at most $n_j-1$ and $\max\{n_j\mid j\}-2$ respectively. The corresponding type II Hermite-Padé approximation problem is
    $$ P_{\vec{n}}(z) f_j(z;\vec{a},\vec{b}) - Q_{\vec{n},j}(z) = \BO(z^{-n_j-1}),\ z\to\infty,\quad j=1,\dots,q,  $$
in terms of polynomials $P_{\vec{n}}(z)$ and $Q_{\vec{n},j}(z)$ of degree at most $\sz{n}$ and $\sz{n}-1$ respectively. The vector of polynomials $(A_{\vec{n},1},\dots,A_{\vec{n},q})$ and the polynomial $P_{\vec{n}}$ are precisely the type~I and type II multiple orthogonal polynomials associated with the moment generating functions $(f_1(z;\vec{a},\vec{b}),\dots,f_q(z;\vec{a},\vec{b}))$, i.e. with weights $(w_1(z;\vec{a},\vec{b}),\dots,w_q(z;\vec{a},\vec{b}))$ as in Sections~\ref{J_W} and \ref{B_W}. Explicit expressions for the other polynomials and the errors are then also known.
\medbreak

We will focus on the type I Hermite-Padé approximation problem. In order to generate interesting examples, it will be more convenient to consider the $q$ functions $f_j^\ast(z;\vec{a},\vec{b}) = f_0(1/z;\vec{a},\vec{b}+\sum_{i=1}^j \vec{e}_i)/z$. We can move between the two settings using the lemma below. The second relation in this lemma can be used to obtain a new approximation problem
\begin{equation} \label{DA_I_HP}
    \sum_{j=1}^q A_{\vec{n},j}^\ast(z) f_j^\ast(z;\vec{a},\vec{b}) - B_{\vec{n}}(z) = \BO(z^{-\sz{n}}),\ z\to\infty,
\end{equation}
in terms of polynomials $A_{\vec{n},k}^\ast(z)=\sum_{j=k}^q c_{j,k} A_{\vec{n},j}(z)$ of degree at most $\max\{n_k,\dots,n_q\}-1$.

\begin{lem} \label{DA_HP_conv}
    Suppose that the entries of $\vec{b}$ are pairwise distinct. Then, 
    \begin{itemize}
        \item[i)] $ f_j^\ast(z;\vec{a},\vec{b}) = \sum_{k=1}^j \lambda_{j,k} f_k(z;\vec{a},\vec{b})$ with $\lambda_{j,k} = 1/\prod_{i=1,i\neq k}^j (b_i-b_k)$, 
        \item[ii)] $f_j(z;\vec{a},\vec{b}) =  \sum_{k=1}^j c_{j,k} f_k^\ast(z;\vec{a},\vec{b})$ with $c_{j,k}=\prod_{i=1}^{k-1} (b_i-b_j)$.
    \end{itemize}
\end{lem}
\begin{proof}
The first part follows from a partial fraction decomposition of $1/\prod_{i=1}^j (s+b_i)$. For the second part, we have to verify that 
    $$\frac{1}{s+b_j}= \sum_{k=1}^j \frac{\prod_{i=1}^{k-1} (b_i-b_j)}{\prod_{i=1}^k(s+b_i)}. $$      
We may use the decomposition $(s+b_j)/(s+b_k) = 1 + (b_j-b_k)/(s+b_k)$ to write
    $$ \sum_{k=1}^j \frac{s+b_j}{s+b_k} \prod_{i=1}^{k-1} \frac{b_i-b_j}{s+b_i} = \sum_{k=1}^j \prod_{i=1}^{k-1} \frac{b_i-b_j}{s+b_i} - \sum_{k=1}^j \prod_{i=1}^{k} \frac{b_i-b_j}{s+b_i} = 1,$$
which then immediately leads to the desired result.
\end{proof}

It is also possible to approximate other linear combinations of the $f_j(z;\vec{a},\vec{b})$ using combinations of the type I polynomials $A_{\vec{n},j}(z)$, as long as the resulting $q$ functions are $\C$-linearly independent.
\medbreak

We can use these approximants to prove the following result.

\begin{thm} \label{DA_RES}
    Consider $p<q$. Let $\vec{a}\in(\Q_{>-1})^p,\vec{b}\in(\Q_{>-1})^q$ with all $a_j<b_j$ and let $z\in\Q_0$. Suppose that $a_i-b_j,b_i-b_j\not\in\Z_{\geq 0}$ whenever $i\neq j$. Then $1$ and the constants
    $$ \sum_{k\geq 0} \frac{(\vec{a}+1)_k}{(\vec{b}+1)_{k} \prod_{i=1}^j (b_i+k+1) } z^k,\quad j=1,\dots,q,$$
    are $\Q$-linearly independent. In particular, these constants are irrational.
\end{thm}

In many cases in which some conditions on the parameters are not satisfied, we can still find a similar result as the above. However, we would have to take appropriate combinations of the approximated functions such that their confluent limits are still $\C$-linearly independent (so that, e.g., the same function doesn't appear twice). The Mellin transform of the corresponding type I function would remain invariant under these operations, but a priori it is not clear how the underlying type I polynomials would be affected. Hence, for such cases, a more hands-on approach is required. For example, in the simplest case, where we take $\vec{a}=\vec{0}\in\R^p$ and $\vec{b}=\vec{0}\in\R^q$, we can show the following result.

\begin{thm} \label{DA_RES_RED}
    Consider $p<q$ and let $z\in\Q_0$. Then $1$ and the constants
    $$ \sum_{k\geq 0} \frac{z^k}{(k+1)^j \cdot k!^{q-p}},\quad j=1,\dots,q,$$
    are $\Q$-linearly independent. In particular, these constants are irrational.
\end{thm}

Note that, since the range of $j$ extends past $q-p$, the above implies in particular that $1$ and the constants $\sum_{k\geq 1} z^k/(k^j\cdot k!^{p-q}) $ for $j=0,\dots,p$ are $\Q$-linearly independent.
\medbreak

Some special cases of Theorems \ref{DA_RES} and \ref{DA_RES_RED} have been proven explicitly before. The case $\vec{a}=(),\vec{b}=(0)$ that corresponds to an approximation using Bessel polynomials, leads to an irrationality proof of $\exp(z)$ for $z\in\Q_0$ and was studied by Grosswald in \cite{Grosswald}. In \cite{Rivoal}, Rivoal extended this result to $\vec{a}=(\alpha),\vec{b}=(0,\alpha)$ with $\alpha\in\Q$ and proved $\Q$-linear independence of $1$, $\exp(z)$ and $\mathcal{E}_\alpha(z)=\sum_{m\geq 0} z^m/((m+\alpha+1) \cdot m!) $. The latter implied that at least one of Euler's constant $\gamma$ and Euler-Gompertz' constant $\delta$ is irrational.

In fact, all the cases covered by Theorems \ref{DA_RES} and \ref{DA_RES_RED}, follow from Siegel's theory on $E$-functions (see, e.g., \cite{ShidlovskiiBook} for an introduction), which was subsequently improved in \cite{Shidlovskii, Chudnovsky, Zudilin}; the matrix Pearson equation for the underlying weights (see Proposition \ref{B_W_MPE}) can be lifted to a system of differential equations for the approximated functions. However, the approximants that are used to prove this are only known implicitly. In our approach, they are explicit and correspond to Hermite-Padé approximants for certain hypergeometric moment generating functions.
\medbreak

Note that if we could take $p=q$ in Theorem \ref{DA_RES_RED}, the constants would reduce to values of consecutive polylogarithms $\text{Li}_j(z)$ for $j=1,\dots,q$. In particular, for $z=-1$, we would recover a logarithm $\ln 2$ and the zeta values $\zeta(j)$ for $j=2,\dots,q$. In this case, our proof will not be valid anymore and a more careful analysis of the approximants is required; they will grow geometrically and will not have a behavior determined by $n!$ anymore. Typically, this will lead to a condition on the values that $z\in\Q$ can take, see \cite{Nikishin_HP_Polylog} where this case was studied.
\medbreak

The proofs of Theorems \ref{DA_RES} and \ref{DA_RES_RED} will consist of two main parts. First, we will analyze the Diophantine quality of the approximants in \eqref{DA_I_HP}, i.e. the growth and Diophantine properties of the denominator polynomials $A_{\vec{n},j}^\ast(z)$ (that arose as linear combinations of the type I polynomials $A_{\vec{n},j}(z)$ from Theorem \ref{B_I_poly}) and the numerator polynomials $B_{\vec{n}}(z)$. Then, we will determine the approximation quality, i.e. the growth of the error. Afterwards, since the overall quality of the approximants will be good enough, we will be able to apply Nesterenko's criterion for $\Q$-linear independence of real numbers, stated below, in its strongest form (with $\tau_1=\tau_2=N-1$).

\begin{thm}[Nesterenko \cite{NesterenkoCrit}]
    Let $x_1,\dots,x_N\in\R$ and suppose that there exists sequences of integers $(p_{j,n})_{n\in\N}$, $j=1,\dots,N$, and an increasing function $\sigma:\R\to(0,\infty)$, with $\lim_{t\to\infty} \sigma(t) = \infty$ and $\limsup_{t\to\infty} \sigma(t+1)/\sigma(t) = 1$, such that
    \begin{itemize}
        \item[i)] $\max\{p_{j,n} \mid j=1,\dots,N\} \leq e^{\sigma(n)}$,
        \item[ii)] $ c_1 e^{-\tau_1 \sigma(n)} \leq \left| \sum_{j=1}^N p_{j,n}x_j \right| \leq c_2 e^{-\tau_2 \sigma(n)}$ for some $c_1,c_2,\tau_1,\tau_2>0$.
    \end{itemize}
    Then $\dim_\Q\text{span}_\Q\{x_1,\dots,x_N\} \geq (\tau_1+1)/(1+\tau_1-\tau_2).$
\end{thm}

We will start by analyzing the Diophantine quality of the approximants in the setting of Theorem \ref{DA_RES}.

\begin{lem}
    (Diophantine quality) Consider $p<q$. Let $\vec{a}\in(\Q_{>-1})^p,\vec{b}\in(\Q_{>-1})^q$ with $a_j<b_j$ for $j=1,\dots,p$ and let $z\in\Q_0$. Suppose that $a_i-b_j,b_i-b_j\not\in\Z_{\geq 0}$ whenever $i\neq j$. Take $\vec{n}=(n,\dots,n)\in\N^q$. Then $A_{\vec{n},j}(z) \leq C^n n!^{q-p} $ as $n\to\infty$ for some $C=C(z;\vec{a},\vec{b})>0$. There exists a sequence of integers $(D_n)_{n\in\N}$ and $D=D(z;\vec{a},\vec{b})>0$ such that $D_n A_{\vec{n},j}(z) \in\Z $ and $D_n\sim D^n$ as $n\to\infty$. As a consequence, similar properties hold for $A^\ast_{\vec{n},j}(z)$. It also holds for $B_{\vec{n}}(z) \Gamma(\vec{b}+1)/\Gamma(\vec{a}+1) $.    
\end{lem}
\begin{proof}
We may use a similar idea as in the proof of \cite[Lemma 2]{ShidlovskiiBook} in which one shows that certain $_{p}F_{q}$-series with $p<q$ are $E$-functions. It was shown that $(a)_k/(b)_k \sim k^c$ as $k\to\infty$ and the least common denominator of $(a)_k/(b)_k$ for $k=0,\dots,n$ is of the order $O(d^n)$ as $n\to\infty$. A suitable representation for $A_{\vec{n},j}(z)$ to analyze can be found in Theorem~\ref{B_I_poly}. Since there is always a deficit of $q-p$ Pochhammer symbols in the denominators of the terms of $A_{\vec{n},j}(z)$, the desired properties immediately follow. The same then holds for $A_{\vec{n},j}^\ast(z)$ due to Lemma \ref{DA_HP_conv} ii). An appropriate representation for $B_{\vec{n}}(z)$ follows from the proof of Theorem \ref{B_I_poly}. After using formula \eqref{J_I_poly_IE} to expand $I_{J,L}(z)$ in the error \eqref{J_I_poly_EE}, we find
    $$B_{\vec{n}}(z) = \frac{\Gamma(\vec{a}+1)}{\Gamma(\vec{b}+1)} \sum_{J=1}^r \sum_{K=0}^{n_J-1} P_{\vec{n},J}[K] \sum_{k=0}^{K-1} \frac{(\vec{b}-b_J-K)_{k+1}}{(\vec{a}-b_J-K)_{k+1}} \frac{1}{K-k+b_J} z^{k+1}$$
and we can repeat the same argument.
\end{proof}

A similar result also holds in the simpler setting of Theorem \ref{DA_RES_RED}.

\begin{lem}
    (Diophantine quality) Consider $p<q$ and set $\vec{a}=\vec{0}\in\R^p,\vec{b}=\vec{0}\in\R^q$. Suppose that $z\in\Q_0$ and take $\vec{n}=(n,\dots,n)$. Then $A^\ast_{\vec{n},j}(z) \leq C^n n!^{q-p} $ as $n\to\infty$ for some $C=C(z;\vec{a},\vec{b})>0$. There exists a sequence of integers $(D_n)_{n\in\N}$ and $D=D(z;\vec{a},\vec{b})>0$ such that $D_n A^\ast_{\vec{n},j}(z) \in\Z $ and $D_n\sim D^n$ as $n\to\infty$. It also holds for $B_{\vec{n}}(z)$.
\end{lem}
\begin{proof}
We will use the same idea as in the proof of Theorem \ref{B_I_poly} to obtain expressions for the polynomials $A^\ast_{\vec{n},j}(z)$ in \eqref{DA_I_HP}. The error in the approximation problem is
    $$\sum_{j=1}^q A^\ast_{\vec{n},j}(z) f_j(z) - B_{\vec{n}}(z) =  \sum_{k\geq 0} \hat{F}_{\vec{n}}(k+1) \frac{1}{z^{k+1}}, \quad f_j(z) = \sum_{k\geq 0} \frac{1}{(k+1)^j k!^{q-p}} \frac{1}{z^{k+1}}. $$
We will analyze the polynomials $A^\ast_{\vec{n},j}(z)$ by working out the right-hand side of the above. We know from Theorem \ref{B_MT} that 
$$\hat{F}(s) = \frac{(1-s)_{qn-1}}{\Gamma(s)^{q-p} (s)_n^{q}},\quad s\in\Z_{\geq 1} .$$
The ratio $(1-s)_{qn-1}/(s)_n^q$ can be seen as the Mellin transform of the (limiting) type~I Jacobi-Piñeiro polynomial $G_{\vec{n}}(x) = \sum_{j=1}^q P^\ast_{\vec{n},j}(x) (\ln x)^{j-1}$ with $\deg P^\ast_{\vec{n},j} \leq n-1$, see \cite{SmetVA} or \cite{Nikishin_HP_Polylog}. It was proven in \cite{Nikishin_HP_Polylog} that there exists $c=c(z)>0$ such that $P^\ast_{\vec{n},j}[k] \leq c^n$, for $0\leq k\leq n_j-1$, as $n\to\infty$ and that there exists a sequence of integers $(d_n)_{n\in\N}$ and $d=d(z)>0$ such that $d_n P^\ast_{\vec{n},j}[k] \in\Z $, for $0\leq k\leq n_j-1$ and $1\leq j\leq q$, and $d_n\sim d^n$ as $n\to\infty$. The associated partial fraction decomposition gives
    $$\frac{(1-s)_{qn-1}}{(s)_n^q} = \sum_{j=1}^q \sum_{l=0}^{n-1} \frac{P^\ast_{\vec{n},j}[l]}{(s+l)^{j}}. $$
Hence, 
$$ \sum_{k\geq 0} \hat{F}_{\vec{n}}(k+1) \frac{1}{z^{k+1}} =  \sum_{j=1}^q \sum_{l=0}^{n-1} P^\ast_{\vec{n},j}[l] \sum_{k\geq 0} \frac{1}{k!^{q-p}} \frac{1}{(k+l+1)^{j}} \frac{1}{z^{k+1}},$$
and it remains to find a way to lower $L$ in 
    $$ I_{J,L}^{[0]}(z) = \sum_{k\geq 0} \frac{1}{k!^{q-p}} \frac{1}{(k+L+1)^{J}} \frac{1}{z^{k+1}},$$
because $I_{J,0}^{[0]}(z) = f_J(z)$. We may do so by making use of the two operations
    $$ I_{J,L}^{[t]}(z) = I_{J-1,L}^{[t+1]}(z) - L \cdot I_{J,L}^{[t+1]}(z),\quad I_{J,L}^{[q-p]}(z) = z I_{J,L-1}^{[0]}(z) - \frac{1}{L^J}, $$
where
$$I_{J,L}^{[t]}(z) = \sum_{k\geq 0} \frac{1}{k!^{q-p}} \frac{1}{(k+1)^{t}} \frac{1}{(k+L+1)^J} \frac{1}{z^{k+1}},\quad I_{0,L}^{[t]}(z) = f_t(z).$$
For a given $L$, the first operation needs to be applied at most $q-p$ times to reach $t=q-p$ which in total leads to an additional factor that is at most $L^{q-p}$. We can then decrease $L$ by $1$ by making use of the second operation, doing so adds a factor $z$. Repeating this $L$ times leads to an additional factor that is at most $z^L L!^{q-p}$. Since we have to do this for every $1\leq L\leq n-1$, the first part of the lemma follows. For the second part, we observe that for a given $L$, the first operation doesn't introduce any denominators, while the second one introduces a denominator $L^J$ to the coefficients of the $A_{\vec{n},j}^\ast(z)$. Hence, after reducing $L$ to $0$, we get a denominator that is at most $\text{lcm}(1,\dots,L)^J$. To handle all $L$, we thus require a factor of at most $\text{lcm}(1,\dots,n-1)^{q-p}$ to cancel all denominators in the coefficients of $A_{\vec{n},j}^\ast(z)$. As $z\in\Q$, it then remains to multiply with $\text{denom}(z)^{n-1}$.
\end{proof}

Since the error in the approximation problem only depends on the Mellin transform of the type I functions (and not on its precise decomposition), the approximation quality can be derived for both cases simultaneously.

\begin{lem}
    (Approximation quality) Consider $p<q$. Let $\vec{a}\in(-1,\infty)^p,\vec{b}\in(-1,\infty)^q$ with $a_j<b_j$ for $j=1,\dots,p$. Take $\vec{n}=(n,\dots,n)$. Then there exists $E=E(z;\vec{a},\vec{b})>0$ such that
    $\sum_{k\geq 0} \hat{F}_{\vec{n}}(k+1) /z^{k+1} = E^{n(1+o(1))} n!^{-q(q-p)},\ n\to\infty. $
\end{lem}
\begin{proof}
It follows from Theorem \ref{B_MT} and the orthogonality conditions that
    $$\sum_{k\geq 0} \hat{F}_{\vec{n}}(k+1) \frac{1}{z^{k+1}} = z^{-qn} \sum_{k\geq 0} \hat{F}_{\vec{n}}(qn+k) z^{-k},\quad \hat{F}_{\vec{n}}(s) = \frac{\Gamma(s+\vec{a})}{\Gamma(s+\vec{b}+n)} (1-s)_{qn-1}.$$
Hence,
$$ \sum_{k\geq 0} \hat{F}_{\vec{n}}(k+1) \frac{1}{z^{k+1}} = (-z)^{-qn} \frac{(1)_{qn-1}(\vec{a}+1)_{qn-1}}{(\vec{b}+1)_{(q+1)n-1}} \frac{\Gamma(\vec{a}+1)}{\Gamma(\vec{b}+1)} {}_{p+1}F_q \left( \begin{array}{c} qn,\vec{a}+qn \\ \vec{b}+(q+1)n \end{array}; \frac{1}{z} \right). $$
The factor in front of the hypergeometric series is asymptotic to $E^{n(1+o(1))} n!^{-q(q-p)}$ as $n\to\infty$ and 
$$\lim_{n\to\infty} {}_{p+1}F_q \left( \begin{array}{c} qn,\vec{a}+qn \\ \vec{b}+(q+1)n \end{array}; \frac{1}{z} \right) = 
\begin{dcases} \sum_{k\geq 0} \frac{q^{(p+1)k}}{(q+1)^{qk}} \frac{z^{-k}}{k!} = \exp( q^q/(q+1)^q z^{-1} ), \text{ if } p=q-1, \\
1, \text{ if } p<q-1,\end{dcases}$$
hence the stated result follows.
\end{proof}

\section*{Acknowledgements}

I would like to thank Walter Van Assche for the many discussions on the topics presented here and for providing feedback on the drafts, Hélder Lima for discussing the results about the zeros of the multiple orthogonal polynomials, Andrei Mart\'\i nez-Finkelshtein for explaining the connection between hypergeometric polynomials and the finite free multiplicative convolution of polynomials, Arno Kuijlaars for providing feedback on the section about random matrices and Wadim Zudilin for suggesting to investigate explicit Diophantine approximation of the constant $\sum_{k\geq 1} 1/(k\cdot k!)$ and some helpful comments.

\end{document}